\documentclass[12pt]{amsart}
\usepackage{amssymb}
\usepackage{epsfig,amsfonts,amssymb,amsthm}
\usepackage{amsmath}
\theoremstyle{plain}
\newtheorem{theorem}{Theorem}[section]

\newtheorem{corollary}[theorem]{Corollary}
\newtheorem{lemma}[theorem]{Lemma}

\theoremstyle{definition}

\theoremstyle{remark}
\newtheorem{rem}{Remark}[section]

\newcommand{\ba}{\begin{eqnarray}}
\newcommand{\be}{\begin{equation}}
\newcommand{\ea}{\end{eqnarray}}
\newcommand{\ee}{\end{equation}}
\newcommand{\benn}{\begin{equation*}}
\newcommand{\eenn}{\end{equation*}}

\def\ve{\varepsilon}

\def\im{\mbox{Im}}

\numberwithin{equation}{section}

\title{On  linear water wave problem in the presence of a critically submerged body}
\author{I.V. Kamotski, V.G. Maz'ya}
\address  {Department of Mathematics, 
University College London, Gower Street, London, WC1E 6BT 
}
\email{i.kamotski@ucl.ac.uk}
\address{Department of Mathematics, Link\"{o}ping University,
  SE-581 83 Link\"{o}ping, Sweden}
\email{vlmaz@mai.liu.se }
\begin{document}

\maketitle

\bigskip

\begin{abstract}

We  study the problem of propagation of linear water waves in a deep water in the presence of a critically submerged body (i.e. the body touching the water surface). Assuming uniqueness of the solution in the energy space, we prove the existence of the solution which satisfies the radiation conditions at infinity as well as, additionally, at the cusp point where the body  touches the water surface. This solution is obtained by the limiting absorption procedure.

Next we introduce a relevant scattering matrix and analyse  its properties. Under a geometric condition introduced by Maz'ya, see \cite{M1}, we show that the method of multipliers applies to cusp singularities, thus proving a new important property of the scattering matrix, which may be interpreted  as the absence of a version of ``full internal reflection".  This property also allows us to  prove uniqueness and existence of the solution in the functional spaces  $H^2_{loc}\cap L^\infty $ and $H^2_{loc}\cap L^p $,  $2<p<6$,  provided  a spectral parameter in the boundary conditions on the surface of the water is large enough.

 This description of the solution does not rely on the radiation conditions or the limiting absorption principle.
 This is   the first result of this type known to us in the theory of linear wave problems in unbounded domains.

\end{abstract}


\newpage
\section{Introduction.}

We study the problem of propagation of linear water waves in a
domain $\Omega$, which represents  water of infinite depth in the
presence of a critically submerged body $\widetilde{\Omega}$. Let us
describe the domain $\Omega$. We fix  a Cartesian system
$x=(x_1,x_2)$ with the origin $O$ and consider a bounded domain
$\widetilde{\Omega}\subset
\mathbb{R}^2_+=\{(x_1,x_2)\in\mathbb{R}^2:x_2>0\}$ (as usual in the
water wave theory, we assume that the  axis $x_2$ points downwards).
We assume that $S:=\partial \widetilde{\Omega}$ is smooth and
touches the water surface $\Gamma:=\{x_2=0\}$ only at the origin
$O$. Further we define $\Omega:=\mathbb{R}^2_+\setminus
\overline{\widetilde{\Omega}}$ and set $\Omega_\tau:=\Omega\cap
\{|x_1|<\tau, \,\, x_2<\tau\}$, where $\tau$ is a small positive
number. We assume that $\Omega_\tau$ coincides with the set \be
\{x:|x_1|<\tau, \ 0<x_2<\phi(x_1)\} ,\ee where $\phi$ is a function
from $C^2[-\tau,\tau]$, such that \be \phi(0)=\phi'(0)=0 ,
\label{ass} \ee and \be \kappa:=\phi''(0)>0 \ . \ee Moreover, let
 $\phi$ be strongly decreasing on $(-\tau,0)$ and strongly increasing on $(0,\tau)$.
 The governing equations are the following:
\be
    \Delta u=f ,\ \text{in}\  \Omega, \label{w1}
\ee \be
    \partial_n u=g_1, \ \text{on}\   S, \label{w2}
\ee \be
    \partial_n u-\nu u=g_2,\ \text{on}\    \Gamma, \label{w3}
\ee where $n$ is the external normal to $\Omega$, $\nu>0$ is a fixed spectral parameter and $f,g_1,g_2$ are  given
functions.

The linear water waves problems for fully submerged bodies in deep water (i.e. when the  body does not touch the water surface) had been studied extensively, see e.g. \cite{J0}-\cite{M1}
 ( see also \cite{KMV}, for more references).
The presence of a critically submerged body implies that
  the domain $\Omega$ contains two { \it external cusps}.
The problems in domains with cusps were studied from various points of view in   \cite{MP}-\cite{NT20} (see also \cite{MS}, where more references can be found).

Our main condition on $\Omega$ is:

\textbf{Condition 1.} Homogeneous problem \eqref{w1}-\eqref{w3} does
not have non-trivial solutions in the energy space $V(\Omega)=\{u:
\int_\Omega |\nabla u|^2dx+\int_{\partial\Omega}
|u|^2ds<+\infty\}$.

This condition indeed holds for many fully submerged bodies. For example, it is well known, see  \cite{M1}, \cite{KMV}, that the following geometric condition implies the uniqueness for  fully submerged bodies:

\textbf{Condition 2.} Let $n(x)=(n_1(x),n_2(x))$ be the unit  normal to $S$, external to $\Omega$. Then we have
\be\label{Mazusl} x_1(x_1^2-x_2^2)n_1(x)+2x_1^2x_2n_2(x)\geq 0  ,\ \ \ x\in S.
\ee

One of the results of this paper is that Condition 2 still implies uniqueness for the case of critically submerged bodies; in fact we can say more, see Theorems \ref{uniq}, \ref{uniq1} and \ref{uniq2}.

We are interested in  existence of  solutions which satisfy
an  {\it outgoing radiation condition at  infinity} (see \eqref{radcondinf}
below, for the precise definition):
\be \label{waves12} u\sim d^+ e^{-\ i\nu x_1-\nu
x_2},\,\,\, \text{as}\ \  x_1\rightarrow +\infty,\ \text{and } \,\,\,\,\,u\sim d^- e^{ i\nu
x_1-\nu x_2},\,\,\, \text{as}\ \  x_1\rightarrow -\infty,
 \ee
where $d^+$ and $d^-$ are some constants.

If $\widetilde{\Omega}$ is completely submerged (i.e. there is no
cusp, $\overline{\widetilde{\Omega}}\in\mathbb{R}^2_+$ ) the
existence of a solution to \eqref{w1}-\eqref{w3} satisfying
radiation conditions at infinity follows immediately, under the
Condition 1, see \cite{KMV} and references therein. Our situation
is more subtle, because of two reasons: the first reason is purely
technical, namely we cannot directly apply the method of \cite{KMV}
which was based on integral equations, due to the presence of the
cusps.

 The other reason is that, depending on the parameter $\nu$,
 the solutions may be {\it not in }$H^1_{loc}(\overline{\Omega})$. The
situation is in fact even more complicated: there may be  many ``reasonable"
solutions and so we need to select only one. The latter implies that
we need to  additionally employ new
 { \it radiation conditions at the cusp}.
To be more precise, we prove, under suitable conditions on
$f,g_1,g_2$ and assuming  Condition 1, that there is a
unique solution to \eqref{w1}-\eqref{w3} satisfying radiation
conditions at infinity and such that, provided $\nu>\kappa/8$,
    \be
    u\sim
    c_1x_1^{-1/2+i\sqrt{\frac{2\nu}{\kappa}-\frac{1}{4}}},\,\, x_1\rightarrow +0,
    \,\,\,\,\,\,\,
    u\sim
    c_2|x_1|^{-1/2+i\sqrt{\frac{2\nu}{\kappa}-\frac{1}{4}}},\,\, x_1\rightarrow-0.
    \label{j1}
    \ee
    In the case
$\nu<\kappa/8$ we have
    \be
    u\sim
    c_1x_1^{-1/2+\sqrt{\frac{1}{4}-\frac{2\nu}{\kappa}}},\,\, x_1\rightarrow +0,
    \,\,\,\,\,\,\,
    u\sim
    c_2|x_1|^{-1/2+\sqrt{\frac{1}{4}-\frac{2\nu}{\kappa}}},\,\,
    x_1\rightarrow-0.\label{j2}
    \ee
In the above formulae $c_1$ and $c_2$ are some constants. For the case $\nu=\kappa/8$ we have the same expressions as in \eqref{j1} but Condition 1 needs to be modified, see Condition 1$'$ in Section 3.

Let us mention that the radiation conditions for the water wave problems in the finite geometry have been studied in \cite{N10} and \cite{KaM}.

The presence of radiation conditions both at  infinity and at the origin presents  new challenges. In particular, we need to employ  to this end a non-standard version of {\it
limiting absorption principle}, cf. e.g. \cite{T}.

 Asymptotic representations \eqref{j1}  and \eqref{j2} show, in
particular, that if $\nu<\kappa/8$ then the solution is in
the space $H^1_{loc}(\overline{\Omega})$. In the case $\nu\geq \kappa/8$
the solution does not belong to $H^1_{loc}(\overline{\Omega})$, in general.
Moreover,  in the latter case there are many solutions with similar
type of behaviour, we show however the condition  \eqref{j1} fixes the unique
one.

Expressions \eqref{waves12} and \eqref{j1} can be   interpreted as ``outgoing waves'', and their complex conjugates as ``incoming waves''. This introduces, for $\nu>\kappa/8$ , a $4\times 4$ { \it scattering matrix}, describing the relation between the incoming and outgoing solutions at both infinities but also, at two cusps.

We study properties of this scattering matrix and show that, apart from the standard ones of unitarity and symmetry, it has more subtle ``block properties'', see Theorem 4.3.
The latter ensures in particular that any combination of waves incoming from the infinities will at least partially ``scatter in the cusps'' (and visa versa). This may be interpreted as the absence of an analogue of full internal reflection (i.e. of  ``infinity to infinity'' or of ``cusps to cusps'' scattering).

The crucial ingredient for establishing the above properties of the scattering matrix is the uniqueness Theorem 4.4. roughly in the class of arbitrary combinations of the cusp incoming and outgoing waves. (In fact in the class of functions with arbitrary inverse polynomial growth at the cusp). We prove this by showing that, under Condition 2, the method of multipliers  see e.g. \cite{M1}, \cite{KMV}, surprisingly, works also in the presence of functions singular at the cusp.

Moreover, the properties of the scattering matrix allow us to establish
the uniqueness and existence results for the problem \eqref{w1}-\eqref{w3} in  various  functional spaces.
We prove that if $\nu>\kappa/8$, and $f, g$ are regular enough and have a compact support, then there exists a unique solution of  problem \eqref{w1}-\eqref{w3}
in the space $ H^2_{loc}(\overline{\Omega})\cap L^\infty(\Omega)$. Under the same conditions we also establish the existence and uniqueness in the space $ H^2_{loc}(\overline{\Omega}\setminus O)\cap L^p(\Omega),\ \ p\in (2,6) $. The former may be interpreted as a solution with no waves either incoming or outgoing to the cusps (hence bounded), and the latter with no similar waves either from or to infinity (hence localised solution in some sense).
In particular these   spaces of functions do not differentiate between the incoming and the outgoing waves and the radiation conditions are not employed anymore.

The paper is organised as follows: in Section 2 we consider
the problem without a submerged body and derive some useful estimates
which are employed in Section 3. There we prove the existence of the
solution of \eqref{w1}-\eqref{w3} in the space of  functions with the radiation conditions, using the
limiting absorption principle. In the last section we introduce the
 scattering matrix for the problem \eqref{w1}-\eqref{w3},
prove some of  its properties and establish  the uniqueness and existence results for the problem \eqref{w1}-\eqref{w3} in  various spaces of functions without radiation conditions.


\section{Problem in $\mathbb{R}_+^2$.} Here we consider an
auxiliary problem in the entire half-space:
  \be
    \Delta u=f ,\ \text{in}\  \mathbb{R}_+^2, \label{aa10}
\ee \be
    \partial_n u-\nu u=0,\ \text{on}\    \Gamma. \label{aa30}
\ee We are interested in the solutions which satisfy the following
{\it radiation condition at  infinity}: $u$ can be represented as a
sum of two outgoing waves and of a function  decaying at infinity.
To make this more precise we define the \textit{outgoing waves at  infinity}, e.g.:
 \be \label{qwe}u_1^-(x)= \chi(x_1)e^{-i\nu x_1-\nu x_2},\ \ u_2^-(x)= \chi(-x_1)e^{i\nu x_1-\nu x_2},
 \ee
where $\chi$ is the cut-off function, such that \be\label{chi}
 \chi\in C^\infty(\mathbb{R}), \ \ \chi(t)=0\  \text{ for}\  t<N,\ \,
 \chi(t)=1\  \text{for}\  t>2N,
\ee and $N$ is a fixed positive number. (Physically, function
$u_1^-$ represents an outgoing wave  moving to the right,
respectively  the outgoing wave $u_2^-$ moves to the left.) Then we
say that $u$ satisfies {\it radiation condition at   infinity}, see
e.g. \cite{KMV}, if \be u=c_1u_1^- +c_2u_2^- +\tilde{u}, \text{in}\
\Omega\setminus B_N,\ \ |\tilde{u}|+|x||\nabla
\tilde{u}|=O(|x|^{-1}) ,\  \text{as}\  |x|\rightarrow \infty,
\label{radcondinf}\ee where $c_1$ and $c_2$ are some constants and $B_N=\{x: \|x\|<N\}$.

The existence of a solution  which satisfies radiation conditions
(under certain assumptions on $f$) is well-known, see e.g.
\cite{KMV}. Below we discuss  relation of this solution to the
limiting absorption principle and derive some useful estimates which
we will apply in the next section.

Consider now the problem with a small absorption described by
$\ve>0$: \be
    \Delta u_\ve-i\ve u_\ve=f ,\ \text{in}\  \mathbb{R}_+^2, \label{aa1}
\ee \be
    \partial_n u_\ve-\nu u_\ve=g_2,\ \text{on}\    \Gamma. \label{aa3}
\ee In order to describe precisely a  solution  of \eqref{aa1},
\eqref{aa3} we introduce the following spaces. Denote $\langle x_j \rangle=(1+x_j^2)^{1/2}$, $j=1,2$, and let, for real $\beta,\gamma$ and
$l=0,1,..$, for  relevant domain $\Theta$: \be
W^l_{\beta,\gamma}(\Theta)=\big\{ u: \sum_{|\delta|\leq
l}\int_\Theta e^{2\beta \langle x_1 \rangle} \langle
x_1\rangle^{2\gamma} \langle x_2\rangle^{2}|\nabla^\delta u|^2dx
<\infty \big\} \ , \ee \be \dot{H}^l(\Theta)=\big\{ u:
\sum_{1<|\delta|\leq l}\int_\Theta |\nabla^\delta u|^2dx
+\int_\Theta\langle x_2\rangle^{-2}|u|^2dx<\infty \big\} \ , \ l\geq 1\, ,\ee
 with the corresponding definitions of the norm and of the trace spaces. For the case $\Theta=\mathbb{R}_+^2$ we omit the dependence on the domain in the notations.

Application of the Fourier transform with respect to $x_1$ and shift
of the contour of integration (see e.g. \cite{KMR1}, \cite{NP} )
yields the following result:

\begin{lemma} \label{l77} \label{l55}Let $\beta> 0$ and  $\ve$ be
such that, $\beta> -\text{Im}(\nu^2-i\ve)^{1/2}$ and
$\frac{\ve^{1/2}}{\sqrt{2}}> -\text{Im}(\nu^2-i\ve)^{1/2}$. Suppose
further that $f\in W^0_{\beta,\gamma}$ and $g_2\in
W^{1/2}_{\beta,\gamma}(\Gamma)$,  and $\gamma \in \mathbb{R}$. Then
there exists a unique solution $u_\ve\in W^2_{0,\gamma}$ of
the problem \eqref{aa1}-\eqref{aa3}, and the following
representation holds
\be \label{rep000} u_\ve=b_1^\ve
U_1^\ve+b_2^\ve U_2^\ve + \tilde{u}_\ve, \ \text{where} \
\tilde{u}_\ve \in W^2_{\beta^*,\gamma}\, , \
\beta^*<\min\{\beta, \frac{\ve^{1/2}}{\sqrt{2}}\}.
\ee
Here
\be
\label{u1} U_1^\ve(x)= \chi(x_1)e^{-i(\nu^2-i\ve)^{1/2} x_1-\nu x_2},
 \ee
 \be \label{u2}
 U_2^\ve(x)= \chi(-x_1)e^{i(\nu^2-i\ve)^{1/2} x_1-\nu x_2},
 \ee
 $ b_1^\ve, b_2^\ve $, are constants,
and \be\label{ocep}
|b_1^\ve|+|b_2^\ve|+\|\tilde{u}_\ve\|_{W^2_{\beta^*,\gamma}}
\leq c
\big(\|f\|_{W^0_{\beta,\gamma}}+\|g_2\|_{W^{1/2}_{\overline{}\beta,\gamma}}\big).
\ee
\end{lemma}
\begin{rem} Clearly we have $U_j^0=u_j^-$, $j=1,2$.
\end{rem}
The constant $c$ appearing in \eqref{ocep} depends on  $\ve$.
The estimate which appears in the next lemma overcomes this
disadvantage.

\begin{theorem} \label{lsem2}Suppose that the conditions of  Lemma \ref{l55} hold, and additionally let us assume that  $\gamma=1$ and $g_2=0$. Then the following estimate holds
\be\label{ocep3} |b_1^\ve|+|b_2^\ve|+\|\tilde{u}_\ve\|_{\dot{H}^2}
\leq c \|f\|_{W^0_{\beta,1}}, \ee where $c$ does not depend
on $\ve$.
\end{theorem}

(Henceforth $c$ is a constant whose value may change from line to line.)

The Theorem \ref{lsem2} is proved in the Appendix.

The above statement allows us to pass to the limit in
\eqref{aa1},\eqref{aa3}, and we have $b_1^\ve\rightarrow b_1$,
$b_2^\ve\rightarrow b_2$, and $\tilde{u}_\ve$ converges to
$\tilde{u}$ weakly in the space $\dot{H}^2$,  as $\ve\rightarrow 0$.
As a result we obtain a solution $u$ to the problem
\eqref{aa10},\eqref{aa30}, which can be represented in the form
$u=b_1 u_1^-+b_2u^-_2+\tilde{u}$.

\section{Critically submerged body}

Consider now the original problem \eqref{w1}-\eqref{w3} with critically submerged body $\widetilde{\Omega}$.
Let us associate with this problem  an ``energy space'': $V=\{u: \int_\Omega |\nabla u|^2dx+\int_{\partial\Omega } |u|^2dS<\infty\}$.
Let us notice that
\be \label{f1}
\int_\Omega (x_2^2+1)^{-1}|v|^2\lesssim  \int_\Omega |\nabla v|^2+\int_{\partial \Omega}|v|^2dS.
\ee
(From now on, $\lesssim$\,  denotes  $\leq c$ with a constant $c$.)
This inequality follows from two obvious inequalities:
\be \int_{\Omega\cap\{|x_1|>N\}} \frac{|u|^2}{x_2^2+1}dx\lesssim  \int_{\Omega\cap\{|x_1|>N\}}
\left|\frac{\partial u}{\partial x_2}\right|^2dx+\int_{\{x_2=0,|x_1|>N\}} |u|^2dx_1,
\ee
\be \int_{\Omega\setminus B_N} \frac{|u|^2}{x_2^2+x_1^2}dx\lesssim  \int_{\Omega\setminus B_N}
\frac{1}{r}\left|\frac{\partial u}{\partial \theta}\right|^2d\theta dr+\int_{\{x_2=0,|x_1|>N\}}|u|^2dx_1,
\ee
and the Friedrichs  inequality
\be    \int_{\Omega\cap B_N} |u|^2dx\lesssim  \int_{\Omega\cap B_N}
|\nabla u|^2dx+\int_{\partial (\Omega\cap B_N)}|u^2|dS,
\ee
which is valid for {\it any} bounded domain, see \cite{M2} \S4.11.1. Here we assume that constant $N$ from the previous section, is such that $\widetilde{\Omega}\subset B_N:=\{x:\|x\|<N\}$.

We are planning to find a solution to the problem \eqref{w1}-\eqref{w3} by employing the principle of  limiting absorption.
In fact we need an absorption in the equation \eqref{w1} and in the boundary conditions \eqref{w2},\eqref{w3}, locally in the neighbourhood of the origin.

Let us fix a cut-off function $\mu\in C_0^\infty(\Gamma)$, such that $\mu(x_1)=1,\ |x_1|<N$ and $\mu(x_1)=0,\
|x_1|>2N$.
Consider now the following problem with a small absorption $\ve\geq 0$:
\be
    \Delta v_\ve-i\ve v_\ve=f ,\ \text{in}\  \Omega, \label{a1}
\ee
\be
    \partial_n v_\ve+i\ve v_\ve=g_1, \ \text{on}\   S, \label{a2}
\ee
    \be
    \partial_n v_\ve-\left(\nu-i\ve\mu \right) v_\ve=g_2,\ \text{on}\    \Gamma, \label{a3}
    \ee
and the corresponding energy space  : $$\mathbf{V}:=\{u: \int_\Omega |\nabla
u|^2dx+\int_{\partial\Omega } |u|^2dS+ \int_\Omega |
u|^2dx<\infty\}.$$

\begin{lemma} \label{exist}Let
$g_1\in L_2(S),\ g_2\in L_2(\Gamma)$ and $ f\in L_2(\Omega)$.
 Then for  any $\ve>0$ there exists a unique solution $v_\ve\in \mathbf{V}$
 of the problem \eqref{a1}-\eqref{a3}.
\end{lemma}
\begin{proof}
 Let us associate with \eqref{a1}-\eqref{a3} a variational
problem: Find $v_\ve$ such that \be  \label{varr}
a_\ve(v_\ve,\varphi):=F(\varphi), \  \forall \varphi\in \mathbf{V}. \ee
Here \be a_\ve(v,\varphi):=\int_\Omega \nabla v\cdot
\overline{\nabla\varphi}\, dx - \int_\Gamma \left(\nu-i\ve\mu(x_1)
\right)v\overline{\varphi}\,dx_1 \ee
$$+i\ve\int_S v\overline{\varphi}\,dS+i\ve \int_\Omega v_\ve\overline{\varphi}\,dx,$$

and \be F(\varphi):=-\int_\Omega  f\overline{\varphi}\, dx
+\int_\Gamma g_2\overline{\varphi}\,dx_1+\int_S
g_1\overline{\varphi}\,dS. \ee
Sesquilinear form
$a_\ve(\cdot,\cdot)$ is clearly continuous and coercive on $\mathbf{V}$,
and
 $F$ is an anti-linear continuous functional on $\mathbf{V}$,
 and the application of  Lax-Milgram lemma, see e.g. \cite{SP},
 gives us a unique solution $v_\ve$ from energy space $\mathbf{V}$.
 Due to ellipticity, the local estimates give us
 $v_\ve \in H^2(\Omega\setminus B_\sigma)$ for any positive $\sigma$.
\end{proof}

 We aim to pass to the limit in \eqref{varr} as $\ve\rightarrow 0$. The main difficulty is
the absence of compactness of embedding of $H^1(\Omega)$ into
$L_2(\partial\Omega)$ and $L_2(\Omega)$ at  infinity, and lack of
compactness of embedding of $H^1(\Omega)$ into $L_2(\partial\Omega)$
in the neighbourhood of the origin due to the presence of the external
quadratic cusps.

 To overcome this problem, we need to employ more
detailed information about properties of the solutions.

We start form the description of $v_\ve$ in right and left
neighbourhoods of the origin:
    \be \Omega_\tau^+:=\Omega_\tau\cap
\{0<x_1\}, \ \ \Omega_\tau^-:=\Omega_\tau \cap \{x_1<0\}  ,
\Upsilon_\tau^\pm:=\partial \Omega_\tau^\pm\cap \partial\Omega \,.
    \ee
In order to describe precisely a  solution  of \eqref{varr},
 we need to introduce the following weighted Sobolev spaces:
 let $\Xi$ be a domain and let $\gamma$ be
real,  $l=0,1,...$, then we define $\mathcal{W}_{\gamma}^l(\Xi)$ and
$V_{\gamma}^l(\Xi)$ as the closures of the set
$C_0^\infty(\overline{\Xi}\setminus O)$ with respect to the norms
    \be
        \|u\|^2_{\mathcal{W}_{\gamma}^l
        (\Xi)}:=
    \label{www}
    \sum_{|\delta|\leq l}\int_{\Xi}
        |x_1|^{4(\gamma-l+|\delta|)}
        |\partial_x^\delta u|^2dx,
    \ee
     \be
        \|u\|^2_{V_{\gamma}^l
        (\Xi)}:=
\label{www66}
    \sum_{|\delta|\leq l}\int_{\Xi}
        |x_1|^{2(\gamma-l+|\delta|)}
        |\partial_x^\delta u|^2dx,
    \ee
respectively, where $\delta\in\mathbb{Z}_+^2$ is the usual multi-index.
Furthermore, for $ l\geq1$ we define
$\mathcal{W}_{\gamma}^{l-1/2}(\partial\Xi)$ and
$V_{\gamma}^{l-1/2}(\partial\Xi)$ as the trace space for
$\mathcal{W}_{\gamma}^l(\Xi)$ and $V_{\gamma}^l(\Xi)$ on the
boundary $\partial\Xi$.

Finally we define the space
$$ \mathcal{V}^2_\gamma(\Omega_\tau^+)=\{u\in V^2_{2\gamma} (\Omega_\tau^+):
P_2u \in \mathcal{W}^2_{\gamma}(\Omega_\tau^+)\},$$ with the norm
\be \|u\|_{\mathcal{V}^2_\gamma(\Omega_\tau^+)}=
\|u\|_{V^2_{2\gamma}(\Omega_\tau^+)}+\|P_2u\|_{\mathcal{W}^2_{\gamma}(\Omega_\tau^+
)}. \ee
Here the projection operator $P_2$ is defined as follows. We represent
$u\in V^2_{2\gamma} (\Omega_\tau^+)$ as
 \be u(x_1,x_2)=u_1(x_1)+u_2(x_1,x_2),\ \ 0<x_1<\tau, \
\ 0<x_2<\phi(x_1), \ee where
$$u_1(x_1)=\phi(x_1)^{-1}\int_0^{\phi(x_1)}u(x_1,x_2)dx_2,$$
and define $$P_1u:=u_1, \,\,\,\  P_2u:=u-u_1:=u_2.$$
We also define fully analogous space $\mathcal{V}^2_\gamma(\Omega_\tau^-)$.

One of  characteristic properties of the above scale of spaces is
the following. If $u\in\mathcal{V}^2_\gamma(\Omega_\tau^+)$ then for
    \be
        \mathcal{L}u:=(\Delta -i\ve) u,
            \ee
and
    \be
        \mathcal{B}u:=\left\{(\partial_n +i\ve )u|_{S\cap \Upsilon_\tau^+},
         (\partial_n -\nu+i\ve)u|_{\Gamma\cap\Upsilon_\tau^+} \right\}
    \ee
we have  $\left\{\mathcal{L}u,\mathcal{B}u
\right\}\in\mathcal{W}_{\gamma}^{0}(\Omega_\tau^+)\times
\mathcal{W}_{\gamma}^{1/2}(\Upsilon_\tau^+)$.
Denote
\be
\label{lambda66}
\lambda_\ve:=2(\nu-2i\ve)/\kappa-1/4.
\ee
The following theorem was proved in \cite{KaM}.
\begin{theorem} (see \cite{KaM}, Theorem 4.3)\label{nfred}
Let  $\gamma\neq
1/2\pm1/2\im \sqrt{\lambda_\ve }$. Suppose that
$u_\ve\in\mathcal{V}^2_\gamma(\Omega_\tau^+)$ is a solution of the
problem
    \be
    \Delta u_\ve-i\ve u_\ve=f ,\ \text{in}\  \Omega_\tau^+, \label{b1}
\ee \be
    \partial_n u_\ve+i\ve u_\ve=g_1, \ \text{on}\   S\cap \Upsilon_\tau^+,
    \label{b2}
\ee
    \be
    \partial_n u_\ve-\left(\nu-i\ve\right) u_\ve=g_2,\ \text{on}\
     \Gamma\cap\Upsilon_\tau^+, \label{b3}
    \ee
where $g:=(g_1,g_2)$ and $(f,g)\in
\mathcal{W}_{\gamma}^{0}(\Omega_\tau^+)\times
\mathcal{W}_{\gamma}^{1/2}(\Upsilon_\tau^+)$. Then  for any $\ve $,
there exists $\delta_0>0$ such that, for any $0<\delta<\delta_0$,
solution $u_\ve$ satisfies the estimate \be \label{llll}
\|u\|_{\mathcal{V}^2_\gamma(\Omega_{\delta/2}^+)}\leq c\left(
\|f\|_{\mathcal{W}^{0}_{\gamma}(\Omega_{\delta}^+)}+
\|g\|_{\mathcal{W}^{1/2}_{\gamma}(\Upsilon_\delta^+)}+
\|u\|_{L_2(\Omega_{\delta}^+\setminus \Omega_{\delta/2}^+)}\right).
 \ee
Here constant $c$ is independent of $f,g$ and $u_\ve$.
 \end{theorem}

Next Theorem from \cite{KaM} describes the structure of the solution.

 \begin{theorem} \label{pred2}
Let  $\gamma$,$\gamma_1$ be real numbers, and $\gamma,\gamma_1\neq 1/2\pm
1/2 \im \sqrt{\lambda_\ve}$,  and $(f,g)\in
\mathcal{W}_{\gamma_1}^{0}(\Omega_\tau^+)\times
\mathcal{W}_{\gamma_1}^{1/2}(\Upsilon_\tau^+)$. Suppose that
$v_\ve\in\mathcal{W}_{\gamma}^{2}(\Omega)$ is a solution of the
 problem \eqref{b1}-\eqref{b3}. Then the solution $v_\ve$ admits
representation
    \be\label{exp}
    v_\ve=c^+Y_\ve^++c^-Y_\ve^-+\widetilde{Y_\ve},\
     \text{in} \ \Omega_{\delta}^+,
    \ee
for sufficiently small positive $\delta$. Here $\widetilde{Y}_\ve\in
\mathcal{V}_{\gamma_1}^{2}(\Omega_{\delta}^+)$,
 $Y_\ve^\pm$ are solutions of homogeneous problem \eqref{b1}-\eqref{b3}, and  $ c^\pm$ are constants.
\end{theorem}

Properties of the special solutions $Y_\ve^\pm$ had been described in \cite{KaM},
see Theorem 4.4,  which can be reformulated in our
context as follows.

\begin{theorem} \label{vv} Suppose that in \eqref{lambda66} $\lambda_\ve\neq 0$. Then there exist  solutions $Y_\ve^+$
and $Y_\ve^-$ of the homogeneous problem
\eqref{b1}-\eqref{b3} in $\Omega_\delta^+$ for small enough  positive $\delta$, such that
\be
\label{igriki}
Y_\ve^\pm(x)=y_1^\pm(x_1,\ve)+y_2^\pm(x,\ve),\
\int_0^{\phi(x_1)}y_2^\pm(x)dx_2=0,\ x_1<\delta,
\ee
where \be \label{yy}
y_1^\pm(x_1,\ve)=x_1^{-\Lambda^\pm_\ve}+\widetilde{y}_1^\pm(x_1,\ve),\
y_1^\pm\in V^2_{2\gamma_\pm}(\Omega_\delta), \
\forall \gamma_\pm>\mp \im \sqrt{\lambda_\ve}/2, \ee
$$
y_2^\pm(x_1,x_2,\ve)=x_1^{2-\Lambda^\pm_\ve}Q_\ve^\pm(z)+\widetilde{y}_2^\pm(x,\ve),\
\widetilde{y}_2^\pm \in
\mathcal{W}^2_{\gamma_\pm}(\Omega_\delta), \ \forall \gamma_\pm>\mp \im \sqrt{\lambda_\ve}/2.
$$
Here
$$z=\frac{x_2}{\phi(x_1)},
$$

\be Q_\ve^\pm(z)=\frac{\kappa}{2}\left(\nu-i\ve\right)\left(\frac{(z-1)^2}{2}-\frac{1}{6}\right)
-\frac{\kappa}{2}\left(\kappa\Lambda^\pm_\ve+i\ve\right)\left(\frac{z^2}{2}-\frac{1}{6}\right),
\label{QQ}\ee
and
$$\Lambda^\pm_\ve = 1/2\pm i \sqrt{\lambda_\ve}.
$$
\end{theorem}

\begin{rem} \label{vip} Let us notice that $Y_\ve^\pm
\in {\mathcal{V}^2_{\sigma_\pm}(\Omega_{\delta}^+)},\ \ \forall
\sigma_\pm>\mp \im \sqrt{\lambda_\ve}/2+1/2$. It will be  useful in what
follows to use another representation for $Y_\ve^\pm$ instead
of \eqref{igriki}, namely \be
Y_\ve^\pm=\textbf{v}^\pm_\ve+\widetilde{\textbf{\textit{v}}}^\pm_\ve,
\ee where \be\label{sss}
\textbf{v}^\pm_\ve(x)=|x_1|^{-\Lambda^\pm_\ve}+|x_1|^{2-\Lambda^\pm_\ve}Q_\ve^\pm\left(x_2/\phi(x_1)\right), \ee
and
$$\widetilde{\textbf{\textit{v}}}^\pm_\ve\in{\mathcal{V}^2_{\gamma_\pm}(\Omega_{\delta}^+)},\ \ \forall \gamma_\pm>\mp \im \sqrt{\lambda_\ve}/2.$$
\end{rem}
\begin{rem} \label{vip2} In the case $\lambda_\ve=0$, i.e $\ve=0$ and $\nu=\kappa/8$, see \eqref{lambda66}, functions  $Y_\ve^\pm$ do still exist and   belong to $ {\mathcal{V}^2_{\sigma_\pm}(\Omega_{\delta}^+)},\ \ \forall
\sigma_\pm>1/2$. We have the following representation for them,
 \be
Y_0^\pm=\textbf{v}^\pm+\widetilde{\textbf{\textit{v}}}^\pm,
\ee where \be\label{sss5}
\textbf{v}^-(x)=|x_1|^{-1/2}+|x_1|^{3/2}Q_0^-\left(x_2/\phi(x_1)\right), \ee
\be\label{sss6}
\textbf{v}^+(x)=|x_1|^{-1/2}\ln{|x_1|}+|x_1|^{3/2}\ln{|x_1|}Q_0^+\left(x_2/\phi(x_1)\right)\ee
$$+|x_1|^{3/2}Q\left(x_2/\phi(x_1)\right),
$$
and
$$\widetilde{\textbf{\textit{v}}}^\pm\in{\mathcal{V}^2_{\gamma}(\Omega_{\delta})},\ \ \forall \gamma>0.$$
\end{rem}
Here $Q_0^\pm$ is defined according to \eqref{QQ} ($Q_0^-=Q_0^+$ in this case)
 and
 \be Q(z)=-\frac{\kappa^2}{2}\left(\frac{z^2}{2}-\frac{1}{6}\right).
 \ee

As we have seen, the structure of the  solution crucially depends on the relation between $\nu$ and $\kappa/8$ (which determines the real part of $\lambda_\ve$ according to \eqref{lambda66}).
Let us start from the most singular case $\nu>\kappa/8$.

Let us check the implications of the above theorems for the solution
$v_\ve$ of the boundary value problem \eqref{a1}-\eqref{a3}, under
assumption that the pair
$$(f,g)\in \mathcal{W}_{0}^{0}(\Omega_\tau^+)\times
\mathcal{W}_{0}^{1/2}(\Upsilon_\tau^+)$$
 and have compact support.
It follows from
Theorems \ref{pred2}, \ref{vv}   and Remark \ref{vip} that
$$v_\ve=d_\ve\textbf{v}^{+}_\ve+c_\ve\textbf{v}^-_\ve+w_\ve,\ \ \text{in}\  \Omega_{\tau}^+, $$
 where $w_\ve\in \mathcal{V}_{0}^{2}(\Omega_{\tau}^+)$ and $c_\ve$, $d_\ve$ are some constants.
 Moreover, Theorem \ref{vv} provides  an information about functions $\textbf{v}^{+}_\ve$ and $\textbf{v}^-_\ve$,  i.e.
$$ \textbf{v}^+_\ve(x)=O\left(x_1^{-1/2-i(\lambda_0-4i\ve\kappa^{-1})^{1/2}}\right),\ \textbf{v}^-_\ve(x)=O\left(x_1^{-1/2+i(\lambda_0-4i\ve\kappa^{-1})^{1/2}}\right),\ \text{as}\  x_1\rightarrow 0,
$$
and consequently $d_\ve=0$ since $v^\ve \in H^1(\Omega)$.

It is easy to see that $w^\ve\in \mathcal{V}_{0}^{2}(\Omega_{\tau}^+)$  solves the problem:
    \be
    \Delta w^\ve=f-
    \big(\Delta -i\ve \big)c_\ve\textbf{v}^-_\ve +i\ve w^\ve,\ \text{in}\
    \Omega_\tau^+, \label{c1}
\ee \be
    \partial_n w^\ve =
    g_1-\big(\partial_n +i\ve\big)c_\ve\textbf{v}_\ve^--i\ve w^\ve, \ \text{on}\   S\cap \Upsilon_\tau^+,
    \label{c2}
\ee
    \be
    \big(\partial_n -\nu \big)w^\ve=
    g_2-\big(\partial_n -\nu+i\ve\big)c_\ve\textbf{v}_\ve^--i\ve w^\ve
    ,\ \text{on}\
     \Gamma\cap\Upsilon_\tau^+. \label{c3}
    \ee
Clearly $R_\ve:=\left(\big(\Delta -i\ve \big)\textbf{v}^-_\ve
,\{\big(\partial_n +i\ve\big)\textbf{v}_\ve^-,\big(\partial_n
-\nu+i\ve\big)\textbf{v}_\ve^-\}\right)\in\mathcal{W}_{\gamma}^{0}(\Omega_\tau^+)\times
\mathcal{W}_{\gamma}^{1/2}(\Upsilon_\tau^+)$,  $\forall \gamma>
2^{-1} \im \sqrt{\lambda_\ve }$. Moreover, for any $\gamma>0$, its norm is uniformly bounded with
respect to $\ve$, due to explicit expression for $\textbf{v}^-_\ve$,
see \eqref{sss}. Next, applying Theorem \ref{nfred} (with $\ve=0$) we obtain,
\be \label{wwww}
\|w_\ve\|_{\mathcal{V}^2_\gamma(\Omega_{\delta/2}^+)}\leq c\left(
\|f\|_{\mathcal{W}^{0}_{\gamma}(\Omega_{\delta}^+)}+
\|g\|_{\mathcal{W}^{1/2}_{\gamma}(\Upsilon_\delta^+)}+
\|w_\ve\|_{L_2(\Omega_{\delta}^+\setminus \Omega_{\delta/2}^+)}+|c_\ve|+\ve\|w_\ve\|_{\mathcal{W}^{0}_{\gamma}(\Omega_{\delta}^+)}\right),
 \ee
for any $\gamma>0$ such that $\gamma\neq
1/2-1/2 \im \sqrt{\lambda_\ve }$. Consequently for any $\gamma>0$, $\gamma\neq 1/2$  and for sufficiently small $\ve>0$, we have,
\be \label{w43}
\|w_\ve\|_{\mathcal{V}^2_\gamma(\Omega_{\delta/2}^+)}\leq c\left(
\|f\|_{\mathcal{W}^{0}_{\gamma}(\Omega_{\delta}^+)}+
\|g\|_{\mathcal{W}^{1/2}_{\gamma}(\Upsilon_\delta^+)}+
\|w_\ve\|_{L_2(\Omega_{\delta}^+\setminus \Omega_{\delta/2}^+)}+|c_\ve|\right).
 \ee
Let us emphasise that constant $c$ in the above formula is independent $\ve, f, g, c_\ve$ and $w_\ve$.

Following the same reasoning, for $v_\ve$ in $\Omega_\tau^-$ we obtain,
\be
      v_\ve(x)= b_\ve \textbf{v}^-_\ve(x) +w_\ve(x), \    x \in \Omega_\tau^-,
\ee
where $b_\ve$ is some constant, $w_\ve\in \mathcal{V}^2_\gamma(\Omega_{\tau}^-)$ for any $\gamma>0$, and
\be \label{w44}
\|w_\ve\|_{\mathcal{V}^2_\gamma(\Omega_{\delta/2}^-)}\leq c\left(
\|f\|_{\mathcal{W}^{0}_{\gamma}(\Omega_{\delta}^-)}+
\|g\|_{\mathcal{W}^{1/2}_{\gamma}(\Upsilon_\delta^-)}+
\|w_\ve\|_{L_2(\Omega_{\delta}^-\setminus \Omega_{\delta/2}^-)}+|b_\ve|\right),
 \ee
where constant $c$  is independent $\ve, f, g, b_\ve$ and $w_\ve$.

Now let us consider $v_\ve$ in $\Omega\setminus B_{2N}$. Using Theorem \ref{lsem2} (we choose $N$ such that $\text{supp} \ g \subset (-N,N)$ ) we obtain,
\be
    v_\ve=b_1^\ve
U_1^\ve+b_2^\ve U_2^\ve + \widetilde{w}_\ve,
\ee
and
\be\label{ocep33} |b_1^\ve|+|b_2^\ve|+\|\tilde{u}_\ve\|_{\dot{H}^2(\Omega\setminus B_{2N})}
\leq c\left(\|f\|_{W^0_{\beta,1}}+
\|w_\ve\|_{L_2(\Omega\cap B_N)}\right),\ee
where constant $c$ in the above formula is independent of $\ve, f, g$ and $w_\ve$.

In the intermediate region, say $\Omega\cap B_{2N}\setminus \Omega_{\delta/2}$, we apply the usual elliptic estimates, yielding
\be \label{5z}\|v^\ve\|_{H^2(\Omega\cap B_{2N}\setminus \Omega_{\delta/2})}\leq c
\left( \|f\|_{L_2(\Omega\cap B_{4N}\setminus \Omega_{\delta/4})}+\|g\|_{H^{1/2}(\partial \Omega)}+\|v^\ve\|_{L_2(\Omega\cap B_{4N}\setminus \Omega_{\delta/4})} \right),
\ee
 where $c$ obviously does not depend on $\ve$.

 Now we are going to combine estimates \eqref{w43},\eqref{w44},\eqref{ocep33} and \eqref{5z}. With this purpose we introduce the following weighted space
\be    \mathcal{H}^2_{\gamma} (\Omega):=\{ \widetilde{v}\in H_{loc}^2(\overline{\Omega}\setminus O):\widetilde{v}\in \mathcal{V}^2_\gamma(\Omega_{\tau}),  \widetilde{v}\in \dot{H}^2(\Omega\setminus B_{\tau/2})\},
\ee
$$   \|\widetilde{v}\|_{\mathcal{H}^2_{\gamma} (\Omega)}^2:=\|\widetilde{v}\|_{\mathcal{V}^2_\gamma(\Omega_{\tau})}^2+\|\widetilde{v}\|_{\dot{H}^2(\Omega\setminus B_{\tau/2})}^2,
$$
and
a space with ``detached asymptotics '':
 \be   v\in \mathbb{H}^2_{\gamma,\ve} (\Omega)\Leftrightarrow v= \sum_{j=1}^4 c_j U_j^\ve+\widetilde{v},\ \ \ \widetilde{v}\in  \mathcal{H}^2_{\gamma} (\Omega),\ \ c_j\in\mathbb{C}, \ j=1,..4.
 \ee
Here $U_1^\ve$ and $U_2^\ve$ are as introduced in \eqref{u1}, and
\be  \label{k5} U_3^\ve(x):=\zeta^+_\tau(x)\textbf{v}^-_\ve(x),\ \ \ U_4^\ve(x):=\zeta^-_\tau(x)\textbf{v}^-_\ve(x),
\ee
 $\zeta^\pm_\tau\in C^\infty (\overline{\Omega}\setminus O)$ are a cut-off functions, such that $\zeta^\pm_\tau=1$ in $\Omega_{\tau/2}^\pm$ and $\zeta^\pm_\tau=0$ in $\Omega\setminus\Omega_{\tau}^\pm$.

 We will refer to $\mathbb{H}^2_{\gamma,0} (\Omega)$ as \textit{space with radiation conditions at the infinity and at the origin}.

Finally, we obtain:
\begin{lemma} \label{got1}Let  $\{f,g\}\in \mathcal{W}^0_{0}(\Omega)\times \mathcal{W}^{1/2}_{0}(\partial \Omega)$ and have compact support.
 Then for  any $\ve>0$ the unique solution $v_\ve\in V$ of the problem \eqref{a1}-\eqref{a3} ensured   by Lemma \ref{exist}, belongs to the space
 $ \mathbb{H}^2_{\gamma,\ve} (\Omega)$ for any $\gamma>0$, in particular
 \be  \label{3l} v_\ve= \sum_{j=1}^4 c_j^\ve U_j^\ve+\widetilde{v}_\ve.
 \ee
 Moreover, for any $\gamma\in(0,1/2)$ and sufficiently small $\ve$, the following estimate is  valid,
\be\label{glavoc} \|\tilde{v}_\ve\|_{\mathcal{V}^2_\gamma(\Omega_{\tau})} +\|\tilde{v}_\ve\|_{\dot{H}^2(\Omega\setminus B_{\tau/2})}
\ee
$$
\leq c \left( \|f\|_{\mathcal{W}^0_{\gamma}(\Omega{})}+\|g\|_{\mathcal{W}^{1/2}_{\gamma}(\partial \Omega)}+\|v_\ve\|_{L_2(\Omega\cap  B_{4N}\setminus \Omega_\tau)}+\sum_{j=1}^4|c_j^{\,\ve}| \right),
$$
where $c$ does not depend on $f,g,c_j$ and $\ve$.
\end{lemma}

 In order to pass to the limit in \eqref{a1}-\eqref{a3}, we need to demonstrate that the ``extra" quantity which appears on the right hand side of \eqref{glavoc}, namely
 $$b_\ve:=\|v_\ve\|_{L_2(\Omega\cap  B_{4N}\setminus \Omega_\tau)}+\sum_{j=1}^4|c_j^{\,\ve}|\, ,$$
 is bounded.

 \begin{lemma} \label{occom} Under Condition 1, we have
\end{lemma}

 \be b_\ve\leq c \left( \|f\|_{\mathcal{W}^0_{\gamma}(\Omega{})}+\|g\|_{\mathcal{W}^{1/2}_{\gamma}(\partial \Omega)} \right),
 \ee
 where $c$ does not depend on $\ve$, $f$ and $g$.
 \begin{proof}
 Let us assume that $b_\ve$ is not bounded, then there exists a subsequence $\ve_{k}$, such that
$b_{\ve_k}>k\left( \|f\|_{\mathcal{W}^0_{\gamma}(\Omega{})}+\|g\|_{\mathcal{W}^{1/2}_{\gamma}(\partial \Omega)} \right),\ k=1,2,...$. Consider a ``normalised" subsequence of $\tilde{v}^\ve$ , $\tilde{u}^{\ve_k}:=\frac{\tilde{v}^{\ve_k}}{ b_{\ve_k}}$ (which we still denote $\tilde{u}^{\ve}$).
The following representation is now valid for this subsequence, cf. \eqref {3l}:
 $$ u_\ve=\tilde{u}_\ve+\sum_{j=1}^4\alpha^\ve_jU_j^{\ve},$$
 and
 \be\label{nnn}\|u_\ve\|_{L_2(\Omega\cap  B_{4N}\setminus \Omega_\tau)}+\sum_{j=1}^4|\alpha_j^{\,\ve}|=1.
 \ee

 Then it follows from \eqref{glavoc}    that we can choose a subsequence (which we still denote $u_\ve$) such that $u_\ve \stackrel{\text{rad}}\rightharpoonup u $ as $\ve\rightarrow 0$. Here the ``weak convergence with radiation conditions", denoted  $\stackrel{\text{rad}}\rightharpoonup$, is understood in the following sense:

 1.   $u$  can be represented as
  $$ u=\tilde{u}+\sum_{j=1}^4\alpha_jU_j^{0},\ \
 $$

 2. $\alpha_j^\ve\rightarrow \alpha_j$, $\tilde{u}_\ve$ converges weakly to $\tilde{u}$ in $ \mathcal{H}^2_{\gamma} (\Omega)$, for any $\gamma>0$,  $U_j^\ve\rightarrow U_j^0$ in $H^2(K)$, where $K$ is any compact set, not containing the singularity point $O$.

 Let us notice that convergence of $U_j^\ve$  to $U_j^0$ follows from explicit expressions for $U_1^\ve$, $U_2^\ve$  see \eqref{u1}, \eqref{u2}, and explicit formulae for  $U_3^\ve$, $U_4^\ve$  see formulae \eqref{k5} and Remark \ref{vip} to Theorem \ref{vv}.
 This allows us to pass to the limit in the  problem \eqref{a1}--\eqref{a3}  (with $u_\ve$ instead of $v_\ve$ and $\frac{1}{ b_{\ve}} \{f,g\}$ instead of $\{f,g\}$) and conclude that $u\in  \mathbb{H}^2_{\gamma,0} (\Omega)$ for any $\gamma>0$ and is a solution of homogeneous problem. Standard trick with integration by parts (see Remark \ref{rem8} below) shows that $\alpha_j=0$, $j=1,..4$,  and consequently $u\in  \mathcal{H}^2_{\gamma} (\Omega)$ for any $\gamma>0$ and in particular $u\in V$ , which, due to Condition 1, implies $u=0$. This contradicts  to \eqref{nnn}, since clearly weak convergence of $\tilde{u}_\ve$ in $  \mathcal{H}^2_{\gamma} (\Omega)$ for any $\gamma>0$ and convergence of $\alpha_j^\ve, \ j=1,..4,$ imply the strong convergence of $u_\ve$ in $L_2$ on compact sets.
\end{proof}

Finally, combining Lemma \ref{exist}, Lemma \ref{got1}, and then treating $v_\ve$ in the same way as $u_\ve$ in Lemma \ref{occom},  we arrive at the following results.

\begin{theorem} \label{toer} Suppose Condition 1 is satisfied and $\nu>\kappa/8$. Assume further that
$\{f,g\}\in \mathcal{W}^0_{0}(\Omega)\times \mathcal{W}^{1/2}_{0}(\partial \Omega)$ and have compact support.
 Then there is  the unique solution of the problem \eqref{w1}-\eqref{w3}, $v\in \mathbb{H}^2_{\gamma,0} (\Omega)$ for any $\gamma>0$,
  in particular
 \be  v= \sum_{j=1}^4 c_j U_j^{0}+\widetilde{v},
 \ee
and for any $\gamma\in(0,1/2)$, the following estimate is  valid,
\be\label{glavoc1} \sum_{j=1}^4|c_j|+\|\tilde{v}\|_{\mathcal{V}^2_\gamma(\Omega_{\tau})} +\|\tilde{v}\|_{\dot{H}^2(\Omega\setminus B_{\tau/2})}
\ee
$$
\leq c \left( \|f\|_{\mathcal{W}^0_{\gamma}(\Omega{})}+\|g\|_{\mathcal{W}^{1/2}_{\gamma}(\partial \Omega)} \right),
$$
where $c$ does not depend on $f$ and $g$.

Moreover, this solution can be obtained as the result of limiting absorption procedure, namely
 $v^\ve\stackrel{\text{rad}}\rightharpoonup v$ where $v_\ve$ is a solution of \eqref{a1}-\eqref{a3}.
\end{theorem}

If $\nu=\kappa/8$  then we apply  arguments as above, employing Remark \ref{vip2} instead of Theorem \ref{vv}. However  in order to obtain the result of the Theorem \ref{toer} we need to employ a stronger condition.

\textbf{Condition 1$'$.} Homogeneous problem \eqref{w1}-\eqref{w3} does not have non-trivial solutions in the space $V'(\Omega)=\{u: u=c^+x_1^{-\frac{1}{2}}\zeta_\tau^++c^-|x_1|^{-\frac{1}{2}}\zeta_\tau^-+\widetilde{u},\ \ c^\pm\in\mathbb{C}, \  \int_\Omega |\nabla \widetilde{u}|^2dx+\int_{\partial\Omega}  |\widetilde{u}|^2ds<+\infty\}$.

The difference comes from the fact that we are able only to prove, in the analogue of Lemma \ref{occom}, that solution of homogeneous problem from  the space with radiation conditions   is actually in energy space $V'(\Omega)$. (For the case  $\nu>\kappa/8$ we were able to deduce that solution is in $V(\Omega)$.) As a  result we conclude,
\begin{corollary} \label{toercrit} If Condition 1$'$ is satisfied then condition $\nu>\kappa/8$ in Theorem \ref{toer} can be relaxed to $\nu\geq\kappa/8$.
\end{corollary}

If $\nu<\kappa/8$,  then there is no need to isolate waves in the cusp.
 Let us introduce the space with ``detached asymptotics" at the infinity only:
 \be   v\in \mathbb{F}^2_{\gamma,\ve} (\Omega)\Leftrightarrow v= \sum_{j=1}^2 c_j U_j^\ve+\widetilde{v}, \ \ \  \widetilde{v}\in \mathcal{H}^2_{\gamma} (\Omega),
 \ee
where $c_j\in\mathbb{C}$.

We have the following analog of Lemma \ref{got1}:
\begin{lemma} \label{got2}Let  $\{f,g\}\in \mathcal{W}^0_{1/2}(\Omega)\times \mathcal{W}^{1/2}_{1/2}(\partial \Omega)$ and have compact support.
 Then for  any $\ve>0$ the unique solution $v_\ve$ of the problem \eqref{a1}-\eqref{a3}, delivered by Lemma \ref{exist}, belongs to the space
 $ \mathbb{F}^2_{1/2,\ve} (\Omega)$, in particular
 \be  v_\ve= \sum_{j=1}^2 c_j^\ve U_j^\ve+\widetilde{v}_\ve.
 \ee
 Moreover, for  sufficiently small $\ve$, the following estimate is  valid,
\be\label{glavoc2} \|\tilde{v}_\ve\|_{\mathcal{V}^2_{1/2}(\Omega_{\tau})} +\|\tilde{v}_\ve\|_{\dot{H}^2(\Omega\setminus B_{\tau/2})}
\ee
$$
\leq c \left( \|f\|_{\mathcal{W}^0_{1/2}(\Omega{})}+\|g\|_{\mathcal{W}^{1/2}_{1/2}(\partial \Omega)}+\|v_\ve\|_{L_2(\Omega\cap  B_{4N}\setminus \Omega_\tau)}+\sum_{j=1}^2|c_j^{\,\ve}| \right),
$$
where $c$ does not depend on $\ve$, $f$ and $g$.
\end{lemma}

Employing similar arguments to the above, we arrive at:
 \begin{theorem} \label{toer2} Suppose Condition 1 is satisfied and $\nu<\kappa/8$. Assume further that
$\{f,g\}\in \mathcal{W}^0_{1/2}(\Omega)\times \mathcal{W}^{1/2}_{1/2}(\partial \Omega)$ and have compact support.
 Then there is  the unique solution of the problem \eqref{w1}-\eqref{w3}, $v\in \mathbb{F}^2_{1/2,0} (\Omega)$,
  in particular
 \be  v= \sum_{j=1}^2 c_j U_j^{0}+\widetilde{v},
 \ee
and  the following estimate is  valid,
\be\label{glavoc122} \sum_{j=1}^2|c_j|+\|\tilde{v}\|_{\mathcal{V}^2_{1/2}(\Omega_{\tau})} +\|\tilde{v}\|_{\dot{H}^2(\Omega\setminus B_{\tau/2})}
\ee
$$
\leq c \left( \|f\|_{\mathcal{W}^0_{1/2}(\Omega{})}+\|g\|_{\mathcal{W}^{1/2}_{1/2}(\partial \Omega)} \right),
$$
where $c$ does not depend on $f$ and $g$.

Moreover, this solution can be obtained as the result of limiting absorption procedure, namely
 $v^\ve\stackrel{\text{rad}}\rightharpoonup v$ in the space $ \mathbb{F}^2_{1/2,0}$, where $v_\ve$ is a solution of \eqref{a1}-\eqref{a3}.
\end{theorem}

\begin{rem} Now formulae \eqref{j2} follows from Theorems \ref{pred2} and \ref{vv} .
\end{rem}

Finally let us comment on the applicability of Condition 1. It has been proved in \cite{M1} that in the case of fully submerged body Condition 2 implies the uniqueness. Various examples  of bodies satisfying Condition 2 can  be found  in \cite{KMV}. In particular, Condition 2 is satisfied by ellipses whose major axis is parallel to  the  $x_2$ axis, see \cite{H}.

One can apply the method of \cite{M1} for  the case of critically submerged body. This method is based on multipliers techniques and integration by parts. So we  only need to verify that integration by parts in the neighbourhood of the origin is justified. For the case $\nu>\kappa/8$, this  can easily be seen, as  a solution of homogeneous problem \eqref{w1}-\eqref{w3}, which is in $V$, belongs to the space $\mathcal{H}^2_{\gamma}(\Omega)$ for any $\gamma$, see Theorem \ref{pred2}. This means that the  solution decays quickly (in fact exponentially) in the neighbourhood of the origin  and the integration by parts is  justified. In other words, for the case $\nu>\kappa/8$ Condition 2 implies Condition 1. The same remains true, if $\nu<\kappa/8$. Then solution of homogeneous problem \eqref{w1}-\eqref{w3}, which is in $V$, belongs to the space $\mathcal{H}^2_{\gamma}(\Omega)$ for some $\gamma<1/2$, see Theorem \ref{pred2} and analysis of possible singularities shows that we still can integrate by parts, see
see Theorem \ref{uniq} below for the details.

The critical case $\nu=\kappa/8$ ,where we need to check Condition 1$'$, is more subtle, but still one can prove that  Condition 2 implies Condition 1$'$, see Remark \ref{coluniq}.

\section{Scattering matrix and its properties}
Let us define the usual scattering matrix for  $\nu<\kappa/8$. In this case we need to employ only waves at the infinity. First we need to introduce "incoming waves":
\be \label{qwe2}u_1^+(x)= \chi(x_1)e^{i\nu x_1-\nu x_2},\ \ u_2^+(x)= \chi(-x_1)e^{-i\nu x_1-\nu x_2},
 \ee
 compare the above with \eqref{qwe}.
\begin{theorem} \label{pros} Suppose that $\nu<\kappa/8$ and  Condition 1 is satisfied. Then there exist two linearly independent solutions of homogeneous problem \eqref{w1}-\eqref{w3}, $\eta_j, j=1,2$ ,
such that
\be \eta_j=u_j^{+}+\sum_{n=1}^{2}s_{jn}u_n^{-}+ \tilde{\eta}_j, \label{rrrr}
\ee
where $\tilde{\eta}_j\in \mathcal{H}^2_{1/2}$.  Condition \ref{rrrr} determines $\eta_j$ uniquely.
The scattering matrix $s=(s_{jn})_{j,n=1}^2$ is unitary and $s_{jn}=s_{nj}$, $j,n=1,2$.
\end{theorem}
\begin{proof} The arguments are standard, see e.g. \cite{NP}. Consider for example case $j=1$. We are looking for $\eta _1$ in the form,
\be\eta_1(x)=e^{i\nu x_1-\nu x_2}+\xi_1(x),
\ee
where $\xi_1(x)$ is a solution of the problem \eqref{w1}-\eqref{w3} with $f=0, g_2=0$ and $g_1=-\partial_n e^{i\nu x_1-\nu x_2}|_S$.
The solution to this problem
exists in the space
$ \mathbb{H}^2_{1/2,0} (\Omega)$, due to Theorem \ref{toer2}. In   particular
 \be  \xi_1= \sum_{j=1}^2 c_j U_j^{0}+\widetilde{v},
 \ee
where $\widetilde{v}\in \mathcal{H}^2_{1/2}(\Omega)$. Since $U_j^{0}=u_j^-$ (compare \eqref{u1},\eqref{u2} with \eqref{qwe}), we see that
$\eta_1(x)=e^{i\nu x_1-\nu x_2}+\xi_1(x)$ satisfies \eqref{rrrr} with $s_{11}=c_1$ and $s_{12}=c_2+1$. Clearly this solution is unique.
The same argument applies to $\eta_2$.

Next we verify the properties of scattering matrix $s$. We know that $\eta_j $ solves the problem \be
    \Delta \eta_j=0 ,\ \text{in}\  \Omega, \label{ww10}
\ee \be
    \partial_n \eta_j=0, \ \text{on}\   S, \label{ww20}
\ee \be
    \partial_n \eta_j=\nu \eta_j,\ \text{on}\    \Gamma. \label{ww30}
\ee

Let us  multiply \eqref{ww10} by $\overline{\eta_n}$, $n=1,2$, integrate over $\Omega \cap \{|x_1|<M\} $ and integrate by parts twice
(which is justified since $\eta_j \in \mathcal{H}^2_{1/2}(\Omega_\tau)$ and due to conditions \eqref{radcondinf}). We have:
\be   0=\int_0^{+\infty}\overline{\eta_n}\partial _{x_1}\eta_j |_{x_1=M}dx_2-\int_0^{+\infty}\overline{\eta_n}\partial _{x_1}\eta_j |_{x_1=-M}dx_2
\ee
$$-\int_0^{+\infty}\eta_j\overline{\partial _{x_1}\eta_j} |_{x_1=M}dx_2+\int_0^{+\infty}\eta_j\overline{\partial _{x_1}\eta_j} |_{x_1=-M}dx_2
.$$
Next using \eqref{rrrr} and \eqref{qwe2}, we
pass to the limit as $M\rightarrow +\infty$, and  obtain
\be 0=\delta_{jn} - \sum_{p=1}^2\overline{s_{np}}s_{jp}.
\ee
So $s$ is indeed unitary.

Now the symmetry property $s_{jn}=s_{nj}$, $j,n=1,2$ follows easily, since \eqref{ww10}-\eqref{ww30} is a problem with real coefficients, $s$ is unitary and
\be  u_j^-=\overline{u_j^+},\ \ j=1,2.
\ee
\end{proof}

In the case $\nu>\kappa/8$, there are { \it four}  linearly independent solutions to homogeneous problem \eqref{w1}-\eqref{w3}, viewed as  solutions of a scattering problem.
First we renormalise functions  $U_j^0$ , $j=3,4$ (see \eqref{sss} and \eqref{k5}):
\be      u_3^-(x):=(\omega\kappa)^{-\frac{1}{2}}U_3^0(x)=(\omega\kappa)^{-\frac{1}{2}}\textbf{v}^{-}_0(x)\zeta_\tau^+
=\label{u3-}\ee
$$(\omega\kappa)^{-\frac{1}{2}}|x_1|^{-\frac{1}{2}+\, i\omega}\left(1+x_1^{2}Q_0^-\left(x_2/\phi(x_1)\right)\right)\zeta_\tau^+,
$$
and
\be      u_4^-(x):=(\omega\kappa)^{-\frac{1}{2}}U_4^0(x)=(\omega\kappa)^{-\frac{1}{2}}\textbf{v}^{-}_0(x)\zeta_\tau^-
=
\label{u4-}\ee
$$(\omega\kappa)^{-\frac{1}{2}}|x_1|^{-\frac{1}{2}+ \, i\omega}\left(1+x_1^{2}Q_0^-\left(x_2/\phi(x_1)\right)\right)\zeta_\tau^-,
$$
\be \omega:=\sqrt{\frac{2\nu}{\kappa}-\frac{1}{4}}. \label{ome}\ee
Similarly to \eqref{qwe} we will refer to these functions as\textit{ outgoing waves in the cusps}. Namely $u_3^-$ is the outgoing wave in the right cusp $\Omega_\tau^+$ and $u_4^-$ is the outgoing wave in the left cusp $\Omega_\tau^-$. In a similar way we introduce \textit{incoming waves in the cusps}:
\be      u_3^+(x):=(\omega\kappa)^{-\frac{1}{2}}\textbf{v}^{+}_0(x)\zeta_\tau^+=(\omega\kappa)^{-\frac{1}{2}}x_1^{-\frac{1}{2}-\, i\omega}\left(1+x_1^{2}Q_0^+\left(x_2/\phi(x_1)\right)\right)\zeta_\tau^+(x),
\label{u3+}\ee
\be      u_4^+(x):=(\omega\kappa)^{-\frac{1}{2}}\textbf{v}^{+}_0(x)\zeta_\tau^-=(\omega\kappa)^{-\frac{1}{2}}|x_1|^{-\frac{1}{2}- \, i\omega}\left(1+x_1^{2}Q_0^+\left(x_2/\phi(x_1)\right)\right)\zeta_\tau^-(x),
\label{u4+}\ee
see \eqref{sss}.
\begin{theorem} \label{SSS} Suppose that $\nu>\kappa/8$ and  Condition 1 is satisfied. Then there exist four linearly independent solutions of homogeneous problem \eqref{w1}-\eqref{w3}, $\eta_j, j=1,..,4$,
such that
\be \eta_j=u_j^{+}+\sum_{n=1}^{4}S_{jn}u_n^{-}+ \tilde{\eta}_j, \label{rrrr2}
\ee
where $\tilde{\eta}_j\in \mathcal{H}^2_{\gamma}$, for all $\gamma>0$.  Condition \eqref{rrrr2} determines $\eta_j$ uniquely.
The scattering matrix $S=(S_{jn})_{j,n=1}^4$ is unitary and $S_{jn}=S_{jn},\ \ j,n=1,..,4$.
\end{theorem}
\begin{proof} The proof of the existence of $\eta_1,\eta_2$ follows the arguments   of Theorem \ref{pros}, with reference to Theorem \ref{toer} instead of Theorem \ref{toer2}. As for $\eta_3$ and $\eta_4$, we need to take some more care. Consider for example $\eta_3$.
Let us look for $\eta_3$ in the form
\be\eta_3=(\omega\kappa)^{-\frac{1}{2}} Y_0^+ \zeta_\delta^+ +\xi_3,
\ee
where $\textit{\textbf{v}}^+$ is the function described in Theorem \ref{vv}, $\delta$ is sufficiently small and
$\xi_3(x)$ is a solution of the problem \eqref{w1}-\eqref{w3} with $f=(\omega\kappa)^{-\frac{1}{2}}\Delta \left(Y_0^+\zeta_\delta^+\right)$ and
$$g_1=-(\omega\kappa)^{-\frac{1}{2}}\partial_n\left(Y_0^+\zeta_\delta^+\right)|_S\, , \,\,\,
g_2=-(\omega\kappa)^{-\frac{1}{2}}(\partial_n-\nu)\left(Y_0^+\zeta_\delta^+\right)|_\Gamma\,.$$ Since $Y_0^+$ is a solution of homogeneous problem \eqref{b1}-\eqref{b3} (with $\ve=0$) for small enough $\delta$, we conclude that $(f,g)\in \mathcal{W}_{\gamma}^{0}(\Omega)\times
\mathcal{W}_{\gamma}^{1/2}(\partial \Omega)$  and Theorem \ref{toer} applies. As a result there is a solution of the problem for $\xi_3$ in the space
$ \mathbb{H}^2_{\gamma,0} (\Omega)$ for any $\gamma>0$. In   particular
 \be  \xi_3= \sum_{j=1}^4 c_j U_j^{0}+\widetilde{v},
 \ee
where $\widetilde{v}\in \mathcal{H}^2_{\gamma}(\Omega)$ for any $\gamma>0$. Since $u_j^-=U_j^{0}$, $j=1,2$ and $u_p^-=(\omega\kappa)^{-\frac{1}{2}}U_p^{0}$, $p=3,4$ we see that
$\eta_3(x)=(\omega\kappa)^{-\frac{1}{2}}Y_0^+\zeta_\delta^++\xi_3$ satisfies \eqref{rrrr2} with $s_{3j}=c_j$, $j=1,2$ and $s_{3p}=(\omega\kappa)^{-\frac{1}{2}}c_p$, $p=1,2$. Clearly this solution is unique.
The same argument applies to $\eta_4$.

Next we verify the properties of scattering matrix $S$. We know that $\eta_j $ solves the problem \be
    \Delta \eta_j=0 ,\ \text{in}\  \Omega, \label{ww1}
\ee \be
    \partial_n \eta_j=0, \ \text{on}\   S, \label{ww2}
\ee \be
    \partial_n \eta_j=\nu \eta_j,\ \text{on}\    \Gamma. \label{ww3}
\ee

Let us  multiply \eqref{ww1} by $\overline{\eta_n}$, $n=1,..,4$, integrate over $\Omega \setminus \Omega_\delta \cap \{|x_1|<M\} $ and integrate by parts twice
(which is justified  due to conditions \eqref{radcondinf}). As a result we have:
\be   0=\int_0^{+\infty}\overline{\eta_n}\partial _{x_1}\eta_j |_{x_1=M}dx_2-\int_0^{+\infty}\overline{\eta_n}\partial _{x_1}\eta_j |_{x_1=-M}dx_2
\ee
$$-\int_0^{+\infty}\eta_j\overline{\partial _{x_1}\eta_j} |_{x_1=M}dx_2+\int_0^{+\infty}\eta_j\overline{\partial _{x_1}\eta_j} |_{x_1=-M}dx_2
$$
$$-\int_0^{\phi(\delta)}\overline{\eta_n}\partial _{x_1}\eta_j |_{x_1=\delta}dx_2+\int_0^{\phi(-\delta)}\overline{\eta_n}\partial _{x_1}\eta_j |_{x_1=-\delta}dx_2
$$
$$+\int_0^{\phi(\delta)}\eta_j\overline{\partial _{x_1}\eta_j} |_{x_1=\delta}dx_2-\int_0^{\phi(-\delta)}\eta_j\overline{\partial _{x_1}\eta_j} |_{x_1=-\delta}dx_2.
$$
Passing to the limits as $M\rightarrow +\infty$ and  $\delta\rightarrow 0$, and using  \eqref{rrrr2},  \eqref{u3-}, \eqref{u4-}, \eqref{u3+}, \eqref{u4+}, \eqref{qwe} and \eqref{qwe2} we obtain
\be 0=i\delta_{jn} - i\sum_{p=1}^4\overline{S_{np}}S_{jp},
\ee
so $S$ is unitary. The property $S_{jn}=S_{jn},\ \ j,n=1,..,4$ can be verified in the same way as in Theorem \ref{pros}.
\end{proof}

\begin{rem} \label{rem8} This type of argument with integration by parts implies the fact (which we used in Lemma \ref{occom}) that if $v\in \mathbb{H}^2_{\gamma,0}$ for all $\gamma>0$ and is a solution of homogeneous problem \eqref{w1}-\eqref{w3}, then $v\in \mathcal{H}^2_{\gamma}$ for all $\gamma>0$.
\end{rem}

Next we describe an important non-standard property of the scattering matrix $S$. First we decompose $S$ as follows:
\be
S=\left(
    \begin{array}{cc}
      S_{(1,1)} & S_{(1,2)} \\
      S_{(2,1)} & S_{(2,2)} \\
    \end{array}
  \right),
\ee
where
\be
S_{(1,1)}=\left(
    \begin{array}{cc}
      S_{1,1} & S_{1,2} \\
      S_{2,1} & S_{2,2} \\
    \end{array}
  \right),
  \ \ \
 S_{(1,2)}=\left(
    \begin{array}{cc}
      S_{1,3} & S_{1,4} \\
      S_{2,3} & S_{2,4} \\
    \end{array}
  \right),
\ee
and
\be
S_{(2,1)}=\left(
    \begin{array}{cc}
      S_{3,1} & S_{3,2} \\
      S_{4,1} & S_{4,2} \\
    \end{array}
  \right),
  \ \ \
 S_{(2,2)}=\left(
    \begin{array}{cc}
      S_{3,3} & S_{3,4} \\
      S_{4,3} & S_{4,4} \\
    \end{array}
  \right).
\ee
\begin{theorem} \label{s22} Suppose  Condition 2 is satisfied and  $\nu>\kappa/8$.  Then $\det S_{(1,2)}\neq 0$,  $\det S_{(2,1)}\neq 0$ and absolute values of eigenvalues  of the matrices
 $S_{(2,2)}$ and  $S_{(1,1)}$ are strictly less then 1.
\end{theorem}

This theorem will follow from the following uniqueness result.
\begin{theorem} \label{uniq}Suppose  Condition 2 is satisfied and  $\nu>\kappa/8$. If  $v$ is a solution of the homogeneous problem \eqref{w1}-\eqref{w3} and
$v\in \mathcal{H}^2_\gamma(\Omega)$ for some $\gamma$, then $v\equiv0$.
\end{theorem}
\begin{proof}
Due to Theorems \ref{pred2} and \ref{vv} we have
\be  \label{qqqq}  v=\sum_{j=3}^4 c_j^+u_j^++\sum_{j=3}^4 c_j^-u_j^-+\widetilde{v},
\ee
where $c_j^\pm$, $j=3,4$ are some constants and $\widetilde{v}\in  \mathcal{H}^2_\gamma(\Omega)$ for any $\gamma>0$. Let us consider the real part of $v$, which we   denote $u$. It is  a solution of the homogeneous problem \eqref{w1}-\eqref{w3} and has same structure as \eqref{qqqq}.

We start by recalling the method of multipliers of \cite{M1}, \cite{KMV}, where it was applied for the case of fully submerged bodies.
 Let $Z=(Z_1,Z_2)$ be a real vector field in $\Omega$ with
at most linear growth as $|x|\rightarrow \infty$ and $Z_2(x_1,0)=0$ for
all $x_1$, and $H$ be a constant. The following identity, which can be
found in \cite{KMV}, p.71, can be verified directly:
    $$
        2\{(Z\cdot\nabla u +H u)\Delta u\}=
        2\nabla\cdot\{(Z\cdot\nabla u+H u)\nabla u\}
    $$
    \be
        +(Q\nabla u)\cdot\nabla u-
        \nabla\cdot\left(|\nabla u|^2Z\right)\,
        \label{ident}
    \ee
Here $Q$ is a $2\times 2 $ matrix with components $Q_{ij}=(\nabla\cdot Z-2H)\delta_{ij}-(\partial_i
Z_j+\partial_jZ_i),\ i,j=1,2$.  Let us choose small positive $\delta<\tau$, integrate \eqref{ident} over $\Omega\setminus \Omega_\delta$ and
integrate  by parts:
    \be\label{oooo}
    0=2\int_{\partial \Omega\setminus \partial \Omega_\delta}(Z\cdot\nabla u+H u)\partial_n
     u ds+\int_{\Omega\setminus \Omega_\delta} (Q\nabla u)\cdot\nabla udx-
\ee
$$  \int_{\partial \Omega\setminus \partial \Omega_\delta}|\nabla u|^2(Z\cdot n)ds+A_\delta^++A_\delta^-\, ,
$$
where
  \be  A_\delta^\pm=\mp2\int_0^{\phi( \pm\delta)}(Z\cdot\nabla u+H u)\partial_{x_1} u|_{x_1=\pm\delta}dx_2 \pm \int_0^{\phi( \pm\delta)}|\nabla u|^2 Z_1|_{x_1=\pm\delta}dx_2\, . \label{AAA}
  \ee
  Hence,
    $$
    0=2\int_{\Gamma\setminus \partial \Omega_\delta}(Z\cdot\nabla u+H u)(\partial_{n}-\nu)
     u dx_1+2\nu\int_{\Gamma\setminus \partial \Omega_\delta}(Z\cdot\nabla u+H u)
     u dx_1$$
     $$+2\int_{S\setminus \partial \Omega_\delta}(Z\cdot\nabla u+H u)\partial_n
     u ds+
    $$
     $$
     +\int_{\Omega\setminus \Omega_\delta}(Q\nabla u)\cdot\nabla udx-
        \int_{S\setminus \partial \Omega_\delta}|\nabla u|^2(Z\cdot n)ds+A_\delta^++A_\delta^-=
    $$
    $$
    2\nu\int_{\Gamma\setminus \partial \Omega_\delta}(Z_1\partial_{x_1}u+H u)
     u dx_1
     +$$
     $$\int_{\Omega\setminus \Omega_\delta} (Q\nabla u)\cdot\nabla udx-
        \int_{S\setminus \partial \Omega_\delta}|\nabla u|^2(Z\cdot n)ds+A_\delta^++A_\delta^-=
    $$
    $$
    \nu\int_{\Gamma\setminus \partial \Omega_\delta}(2H -\partial_{x_1} Z_1  )
     |u|^2 dx_1+$$
     \be \label{zvezda123}
\int_{\Omega\setminus \Omega_\delta} (Q\nabla u)\cdot\nabla udx-
        \int_{S\setminus \partial \Omega_\delta}|\nabla u|^2(Z\cdot n)ds +A_\delta^++A_\delta^-+B_\delta^++B_\delta^-.
   \ee
  Here
  \be  B_\delta^\pm=\mp\nu Z_1(\pm\delta, 0)
     u^2(\pm \delta,0).
  \ee
    Following \cite{KMV} (see p.76), we choose
    \be
    Z(x_1,x_2)=\left(x_1\frac{x_1^2-x_2^2}{x_1^2+x_2^2},\frac{2x_1^2x_2}{x_1^2+x_2^2}\right) \ \ \text{ and}
    \ \ H=1/2
    .
    \label{zh}
    \ee
    Then, in particular, the first term in the right hand side of
    \eqref{zvezda123} is equal to zero. Moreover it was also  verified in \cite{KMV}   (see p.76) that  the quadratic form $(Q\nabla u)\cdot\nabla u$ is non--positive. In fact it has been shown in \cite{KMV}, and can be verified by direct inspection, that
    \be    (Q\nabla u)\cdot\nabla u=-\left (2x_1x_2\partial_{x_1} u+ (x_2^2-x_1^2)\partial_{x_2} u\right)^2(x_1^2+x_2^2)^{-2}.
 \label{mazq}   \ee
Finally, Condition 2 insurers that
    \be \label{Mazcon55}(Z\cdot n)\geq 0 \,\,\, \text{ on}\,\,\  S .\ee

    Now  we need to investigate the behaviour of $A_\delta^\pm+B_\delta^\pm$ as $\delta\rightarrow 0$.
Due to \eqref{qqqq}, we have
\be\label{u6}u(x)=a^+x_1^{-1/2}\cos(\omega\ln x_1 +b^+)+O(x_1^{1/2}) , \ \    x\in\Omega_\tau^+, \ \ x_1\rightarrow +0,
\ee
and
$$u(x)=a^-(-x_1)^{-1/2}\cos(\omega\ln (-x_1) +b^-)+O(x_1^{1/2}) , \ \    x\in\Omega_\tau^-, \ \ x_1\rightarrow -0,
$$
where  $a^\pm$ and $ b^\pm$ are some real constants.
Consider for definiteness  $A_\delta^++B_\delta^+$. Then we employ Theorem \ref{vv} and obtain
$$
\partial_{x_1}u(x)=
a^+x_1^{-3/2}\left( -2^{-1}\cos(\omega\ln x_1 +b)-\omega \sin(\omega\ln x_1 +b )\right)$$
$$+O(x_1^{-1/2}), \ x\in\Omega_\tau^+, \ \ x_1\rightarrow +0,
$$
$$
\partial_{x_2} u(x)=O(x_1^{-1/2}), \ x\in\Omega_\tau^+, \ \ x_1\rightarrow +0.
$$
Moreover, we have from \eqref{zh}
 $$Z_1(x)=x_1+O(x_1^2),    \ x\in\Omega_\tau^+, \ \ x_1\rightarrow +0,$$
 and
 $$Z_2(x)=O(x_1^2),    \ x\in\Omega_\tau^+, \ \ x_1\rightarrow +0,$$
and consequently
$$  \left(\big(\partial_{x_1}u(x)\big)^2+\big(\partial_{x_2}(x)u\big)^2\right) Z_1(x)-2(Z(x)\cdot\nabla u(x)+H u(x)) \partial_{x_1}u(x)
$$
$$=- \partial_{x_1}u(x)\big(x_1 \partial_{x_1}u(x)+u(x)\big)+O(x_1^{-1})
$$
$$=-(a^+)^2x_1^{-2}\left( -2^{-1}\cos(\omega\ln x_1 +b)-\omega \sin(\omega\ln x_1 +b )\right)\times
$$
$$\left( 2^{-1}\cos(\omega\ln x_1 +b)-\omega \sin(\omega\ln x_1 +b )\right)+O(x_1^{-1})
$$
$$=(a^+)^2x_1^{-2}\left\{ 4^{-1}\cos^2(\omega\ln x_1 +b^+) -\omega^2 \sin^2(\omega\ln x_1 +b^+)\right\}
+O(x_1^{-1})
$$
$$=(a^+)^2x_1^{-2}\left\{ \left(4^{-1}+\omega^2\right)\cos^2(\omega\ln x_1 +b^+) - \omega^2\right\}
+O(x_1^{-1}),\ \ \ x\in\Omega_\tau^+, \ \ x_1\rightarrow +0.
$$
Next, using
$$\phi(x_1)=\kappa x_1^2/2+O(x_1^3),$$
we get
\be   A^{+}_\delta=2^{-1}(a^+)^2\kappa\left\{ \left(4^{-1}+\omega^2\right)\cos^2(\omega\ln \delta +b^+) - \omega^2\right\}
+O(\delta),\ \ \ \text{as} \ \delta\rightarrow +0.
\ee
For $B_\delta^+$, we have, using \eqref{u6},
$$ B_\delta^+=-\nu (a^+)^2 \cos^2(\omega\ln \delta +b^+)+O(\delta),\ \ \ \text{as} \ \delta\rightarrow +0.
$$
Finally using \eqref{ome}, we obtain
\be  A^{+}_\delta+B_\delta^+=-2^{-1}(a^+)^2\kappa\omega^2
+O(\delta),\ \ \ \text{as} \ \delta\rightarrow +0.
\ee
In the same way we derive,
\be  A^{-}_\delta+B_\delta^-=-2^{-1}(a^-)^2\kappa\omega^2
+O(\delta),\ \ \ \text{as} \ \delta\rightarrow +0.
\ee
Now we can pass to the limit in \eqref{zvezda123}, as $\delta\rightarrow 0$. We get
\be 0=
\int_{\Omega} (Q\nabla u)\cdot\nabla udx-
        \int_{S}|\nabla u|^2(Z\cdot n)ds-2^{-1}\kappa\omega^2\left((a^-)^2 +(a^+)^2\right) .
\label{kkk}\ee
As result we conclude via \eqref{mazq} and \eqref{Mazcon55} that $u\equiv0$.

Applying the same arguments to imaginary part of $v$ we obtain the same result.
\end{proof}
\begin{corollary} \label{coluniq} It follows from the proof of Theorem \ref{uniq} that Condition 2 implies Condition 1$'$.
\end{corollary}
Now Theorem \ref{s22} easily follows.
\begin{proof} (of Theorem \ref{s22})  It follows from the properties of scattering matrix $S$, see Theorem \ref{SSS}, that it is enough to prove only one of the claims of the Theorem \ref{s22}. Let us prove that $\det S_{(2,1)}$ is not zero. If it is not so, then  there is a non zero row $a=(a_1,a_2)$ such that $aS_{(2,1)}=(0,0)$ and consequently the function $u=a_1\eta_3+a_2\eta_4$ is not identically zero and satisfies the conditions of Theorem \ref{uniq}. This delivers a contradiction.
\end{proof}

Theorem \ref{s22} allows us to formulate other existence and uniqueness results.
\begin{theorem} \label{uniq1} Suppose Condition 2 is satisfied and $\nu>\kappa/8$. Assume further that
$\{f,g\}\in L_2(\Omega)\times H^{1/2}(\partial \Omega)$ and have compact supports separated from the origin.
 Then there is  the unique solution of the problem \eqref{w1}-\eqref{w3}, $u$, in the space $ H^2_{loc}(\overline{\Omega})\cap L^\infty(\Omega)$. Moreover,

 \be  \label{uuu} u= \sum_{j=1}^2 c_j^+ u_j^{+}+\sum_{j=1}^2 c_j^- u_j^{-}+\widetilde{u},
 \ee
where $\widetilde{u}\in \mathcal{H}_\gamma^2(\Omega)$, for any $\gamma$ and $c_j^\pm$, $j=1,2$ are constants.
\end{theorem}
\begin{proof}
 Theorem \ref{toer} delivers us the unique solution of the problem \eqref{w1}-\eqref{w3}, $v\in \mathbb{H}^2_{\gamma,0} (\Omega)$ for any $\gamma>0$,
i.e.
 \be  v= \sum_{j=1}^4 d_j u_j^{-}+\widetilde{v},
 \ee
where  $\widetilde{v}\in \mathcal{H}^2_{\gamma} (\Omega)$ for all $\gamma>0$.
Consider the function
\be \label{uuuu} u:=v-a_1\eta_1-a_2\eta_2,
\ee
where a row $a=(a_1,a_2)$  solves
$$ a S_{(1,2)}=(d_3,d_4).
$$
The solution exists due to Theorem \ref{s22}. It is easy to see that $u$ defined by \eqref{uuuu} is a solution of \eqref{w1}-\eqref{w3} and has the structure \eqref{uuu} with
$\widetilde{u}\in \mathcal{H}^2_{\gamma} (\Omega)$ for any $\gamma>0$. The inclusion $\widetilde{u}\in \mathcal{H}^2_{\gamma} (\Omega)$, for  $\gamma\leq 0$, follows from Theorems \ref{pred2} and \ref{vv}. Clearly $u\in H^2_{loc}(\overline{\Omega})\cap L^\infty(\Omega)$.

Let us discus uniqueness. Consider some $u\in H^2_{loc}(\overline{\Omega})\cap L^\infty(\Omega)$ which is a solution of \eqref{w1}-\eqref{w3}. If we additionally know that representation \eqref{uuu} is valid for $u$, then
 uniqueness follows immediately from Theorem \ref{uniq}.

 Let us prove that representation \eqref{uuu} is valid.
It follows from Theorems \ref{pred2} and \ref{vv} that $u\in \mathcal{H}_\gamma^2(\Omega_\tau)$ for any $\gamma$, since $f$ and $g$ have support separated from the origin and $u\in H^2_{loc}(\overline{\Omega})$. Now we need to show that the representation \eqref{uuu} is valid at infinity. But this is the same as to prove that  any solution $v\in H^2_{loc}(\mathbb{R}^2_+)\cap L^\infty(\mathbb{R}^2_+)$ of \eqref{aa10},\eqref{aa30} with compactly supported $f\in L_2(\mathbb{R}^2_+)$ can be represented as:
\be  \label{uuu77} v= \sum_{j=1}^2 c_j^+ u_j^{+}+\sum_{j=1}^2 c_j^- u_j^{-}+\widetilde{v},
 \ee
where  $\widetilde{v}\in \dot{H}^2(\mathbb{R}^2_+)$. Consider  solution $v_1$ of the problem \eqref{aa10},\eqref{aa30} in the space with radiation conditions. Clearly representation \eqref{uuu77} is valid for $v_1$ (in fact coefficients next to outgoing waves are zero) and $v_1\in H^2_{loc}(\mathbb{R}^2_+)\cap L^\infty(\mathbb{R}^2_+)$. Then
$w:=v-v_1\in H^2_{loc}(\mathbb{R}^2_+)\cap L^\infty(\mathbb{R}^2_+)$ is a solution of \eqref{aa10},\eqref{aa30} with zero right hand side and consequently $w$ is a linear combination of functions $e^{-i\nu x_1-\nu x_2}$ and $e^{i\nu x_1-\nu x_2}$. We see that representation \eqref{uuu77} is valid for $w$, $v_1$ and consequently is valid for $v$. This completes the proof.
\end{proof}

In a similar way we prove the next result.

\begin{theorem} \label{uniq2} Suppose Condition 2 is satisfied and $\nu>\kappa/8$. Assume further that
$\{f,g\}\in L_2(\Omega)\times H^{1/2}(\partial \Omega)$ and have compact support separated from the origin.
 Then there is  a unique solution of the problem \eqref{w1}-\eqref{w3} $w$ in the space
 $$ H^2_{loc}(\overline{\Omega}\setminus O)\cap L^p(\Omega),\ \ p\in (2,6). $$
 Moreover, $w\in \mathcal{H}^2_\gamma$,  for any $\gamma>1/2$ and can be represented as
 \be  \label{wwww55} w= \sum_{j=3}^4 c_j^+ u_j^{+}+\sum_{j=3}^4 c_j^- u_j^{-}+\widetilde{w},
 \ee
where $\widetilde{w}\in \mathcal{H}_\gamma^2(\Omega)$ for any $\gamma>0$ and $c_j^\pm$, $j=1,2$, are some constants.
\end{theorem}
\begin{proof}
 Theorem \ref{toer} delivers us the unique solution of the problem \eqref{w1}-\eqref{w3}, $v\in \mathbb{H}^2_{\gamma,0} (\Omega)$ for any $\gamma>0$,
i.e.
 \be  v= \sum_{j=1}^4 d_j u_j^{-}+\widetilde{v},
 \ee
where  $\widetilde{v}\in \mathcal{H}^2_{\gamma} (\Omega)$ for any $\gamma>0$.
Consider the function
\be \label{wwww66} w:=v-a_1\eta_3-a_2\eta_4,
\ee
where the row $a=(a_1,a_2)$  solves
$$ a S_{(2,1)}=(d_1,d_2).
$$
The solution exists due to Theorems \ref{s22}. Then it follows from \ref{SSS}  that $w$ defined by \eqref{wwww66} is a solution of \eqref{w1}-\eqref{w3} and has the structure \eqref{wwww55} with
$\widetilde{w}\in \mathcal{H}^2_{\gamma} (\Omega)$ for any $\gamma>0$.
It follows from Theorem \ref{vv} that $w\in L_p(\Omega_\tau)$ for any $p<6$. On the other hand, \eqref{radcondinf} implies
$w\in L_p(\Omega\setminus \Omega_\tau)$ for any $p>2$. Consequently $w\in H^2_{loc}(\overline{\Omega}\setminus O)\cap L^p(\Omega),\ \ p\in (2,6)$.

Let us discus uniqueness. Consider some $w\in H^2_{loc}(\overline{\Omega}\setminus O)\cap L^p(\Omega),\ \ p\in (2,6)$ which is a solution of \eqref{w1}-\eqref{w3}. If we additionally know that representation \eqref{wwww55} is valid for $w$, then
 uniqueness follows immediately from Theorem \ref{uniq}.

 Let us prove that representation \eqref{wwww55} is valid. It is easy to see, that since $w\in L_p(\Omega_\tau)$ for any $p<6$ and is a solution of \eqref{w1}-\eqref{w3} with $f$ and $g$ having support separated from the origin, $w\in \mathcal{H}_\gamma^2(\Omega)$ for some large $\gamma$.
Then, employing  Theorems \ref{pred2} and \ref{vv}, we conclude that $w\in \mathcal{H}_\gamma^2(\Omega_\tau)$ for any $\gamma>1/2$ and representation \eqref{wwww55} is valid in the neighbourhood of the origin.

Let us consider representation \eqref{wwww55} at infinity. We need to prove that if $w\in H^2_{loc}(\overline{\Omega}\setminus O)\cap L^p(\Omega),\ \ p\in (2,6)$ then $w\in \dot{H}^2(\Omega\setminus \Omega_\tau)$.
 However this is the same as to prove that  any solution $u\in H^2_{loc}(\mathbb{R}^2_+)\cap L^p(\mathbb{R}^2_+)$, $p\in (2,6)$, of \eqref{aa10},\eqref{aa30} with compactly supported $f\in L_2(\mathbb{R}^2_+)$  is in fact in $\dot{H}^2(\mathbb{R}^2_+)$.
 Consider  solution $u_1$ of the problem \eqref{aa10},\eqref{aa30} which has been obtained in Section 2, see discussion following Theorem 2.2. Clearly $u_1$ can be represented as
 $$
 u_1= c_1 u_1^-+c_2 u_2^-+ \widetilde{u}_1,
$$
where
$\widetilde{u}_1\in \dot{H}^2(\mathbb{R}^2_+)$ and $c_1,c_2$ are some constants. Moreover it follows from \eqref{radcondinf} that $u_1 \in L^p(\mathbb{R}^2_+)$, $p\in (2,6)$.
 Then
$w:=u-u_1\in H^2_{loc}(\mathbb{R}^2_+)\cap L^\infty(\mathbb{R}^2_+)$ is a solution of \eqref{aa10},\eqref{aa30} with zero right hand side, and consequently $w$ is a linear combination of functions $e^{-i\nu x_1-\nu x_2}$ and $e^{i\nu x_1-\nu x_2}$. As result, we see that
$$
 u= c_1 u_1^-+c_2 u_2^-+ b_1e^{-i\nu x_1-\nu x_2}+b_2e^{i\nu x_1-\nu x_2}+ \widetilde{u}_1,
$$
where $b_1$ and $b_2$ are some constants. We know that $u\in H^2_{loc}(\mathbb{R}^2_+)\cap L^p(\mathbb{R}^2_+)$, $p\in (2,6)$, on the other hand functions $
 u_1^-(x), \ u_2^-(x),\  e^{-i\nu x_1-\nu x_2}$ and $e^{i\nu x_1-\nu x_2} $
 do not belong to $L^p(\mathbb{R}^2_+)$, $p\in (2,6)$ and are linearly independent. Consequently $u$ coincides with $\widetilde{u}_1$ which is in
$\dot{H}^2(\mathbb{R}^2_+)$.
 This ends the prove.
\end{proof}

Solutions $u$ and $w$, delivered by Theorems \ref{uniq1} and \ref{uniq2}, do not satisfy radiation conditions in general, and cannot be obtained as a result of limiting absorption procedure. However their description is simple, $u$ is bounded and $w$ decays at the infinity. It is worth emphasising that  these results rely on Condition 2.

We conclude with brief remarks on how one can define a scattering matrix for the case $\nu=\kappa/8$.
We follow  constructions  which appeared in \cite{NP} for domains with conical points and \cite{KN1}-\cite{KN3} for periodic media.
First we introduce incoming and outgoing waves in the cusps, for the \textit{threshold } case $\nu=\kappa/8$:
\be      u_3^\pm:=\left(\textbf{v}^{+}\pm i\textbf{v}^{-}\right)\zeta_\tau^+,
\label{u3cr}\ee
\be      u_4^\pm:=\left(\textbf{v}^{+}\mp i\textbf{v}^{-}\right)\zeta_\tau^-,
\label{u4cr}\ee
see \eqref{sss5} and \eqref{sss6}.
Waves $u_1^\pm$ and $u_2^\pm$ are defined according to \eqref{qwe} and \eqref{qwe2}.
Arguing in the same way as for the case $\nu>\kappa/8$, we obtain
\begin{theorem} \label{SSScrit} Suppose that $\nu=\kappa/8$ and  Condition 1$'$ is satisfied. Then there exist four linearly independent solutions of homogeneous problem \eqref{w1}-\eqref{w3}, $\eta_j, j=1,..,4$,
such that
\be \eta_j=u_j^{+}+\sum_{n=1}^{4}a_{jn}u_n^{-}+ \tilde{\eta}_j, \label{rrrr22}
\ee
where $\tilde{\eta}_j\in \mathcal{H}^2_{\gamma}$, for any $\gamma>0$.  Condition \eqref{rrrr22} determines $\eta_j$ uniquely.
\end{theorem}

\appendix
\section{Appendix}

In this appendix we prove Theorem \ref{lsem2} and obtain some results which may be of their own interest.

It is clear that one needs to prove the estimate \eqref{ocep3}  only for small $\ve$, therefore we will assume that $\ve\leq 2\beta\nu$.
The desired estimate for the coefficients $b_1^\ve$ and $b_2^\ve$ follows from the explicit formulae, by multiplying \eqref{aa1} by the solutions of the homogeneous problem \eqref{aa1}, \eqref{aa3} and integrating over $\mathbb{R}_+^2$. As a result, upon a straightforward integration by part we get
\be \label{bbb1} b_1^\ve=
i\nu (\nu^2-i\ve)^{-1/2}
\int_{\mathbb{R}_+^2}f(x)e^{i(\nu^2-i\ve)^{1/2} x_1-\nu x_2}dx,
\ee
\be \label{bbb2} b_2^\ve= i\nu (\nu^2-i\ve)^{-1/2} \int_{\mathbb{R}_+^2}f(x)e^{-i(\nu^2-i\ve)^{1/2} x_1-\nu x_2}dx.
\ee

The remainder $\tilde{u}_\ve\in W_{\beta^*,1}^2$  (see \eqref{rep000} ) solves the problem
  \be
    \Delta \widetilde{u}_\ve- i \ve\tilde{u}_\ve=F ,\ \text{in}\  \mathbb{R}_+^2, \label{naa1}
\ee
\be
    \partial_n \tilde{u}_\ve-\nu \tilde{u}_\ve=0,\ \text{on}\    \Gamma, \label{naa3}
\ee
where
\be  F(x)=f(x)-b_1^\ve [\Delta,\chi(x_1)]e^{-i(\nu^2-i\ve)^{1/2} x_1-\nu x_2}-b_2^\ve [\Delta,\chi(-x_1)]e^{i(\nu^2-i\ve)^{1/2} x_1-\nu x_2},
\ee
and  $[A,B]=AB-BA$ is a commutator.

Clearly via applying Cauchy-Schwartz inequality to \eqref{bbb1} and \eqref{bbb2}, we have the estimate:
\be   \|F\|_{W^0_{\beta,1}}\lesssim \|f\|_{W^0_{\beta,1}}\label{tttt}.
\ee
Thus we need to prove the inequality
\be
\|\tilde{u}_\ve\|_{\dot{H}^2} \lesssim  \|F\|_{W^0_{\beta,\gamma}}.
\label{Aoc}\ee
We are going to apply method of projections to \eqref{naa1}, \eqref{naa3}. This method, in the context of linear water waves, probably goes back to \cite{J}. To this end we represent $\tilde{u}_\ve$ as
\be \tilde{u}(x_1,x_2)=w_1(x_1)e^{-\nu x_2}+ w_2(x_1,x_2), \label{razl0}\ee
where
  \be \int_0^{\infty} w_2(x_1,x_2)e^{-\nu x_2}dx_2=0,\ \forall x_1\in \mathbb{R}.\label{razl}
\ee
Obviously, by the construction of $w_1$ as a projection of $\tilde{u}_\ve$, we have the estimates
\be \|e^{\beta^*\langle x_1 \rangle}w_1\|_{H^2(\mathbb{R})}\leq c \|\tilde{u}_\ve\|_{\mathbb{W}^2_{\beta^*,1}}, \ \ \|w_2\|_{W^2_{\beta^*,1}}\leq c \|\tilde{u}_\ve\|_{W^2_{\beta^*,1}}.
\ee
Similarly we represent $F$ as
\be F(x_1,x_2)=f_1(x_1)e^{-\nu x_2}+ f_2(x_1,x_2), \ \  \int_0^{\infty} f_2(x_1,x_2)e^{-\nu x_2}dx_2=0,\  x_1\in \mathbb{R},
\ee
with estimates
\be \|e^{\beta\langle x_1 \rangle}\langle x_1 \rangle f_1\|_{L_2(\mathbb{R})}\leq c \|F\|_{W^0_{\beta,1}}, \ \ \|f_2\|_{W^0_{\beta,1}}\leq c \|F\|_{W^0_{\beta,1}}.\label{ttttt}
\ee
Then we get by direct inspection decoupled problems for $w_1$ and $w_2$:
\be \partial_{x_1}^2 w_1(x_1) + (\nu^2-i\ve) w_1 (x_1)= f_1(x_1), \ \ x_1\in \mathbb{R},\label{hhh}
\ee
and
\be
    \Delta w_2-i\ve w_2 =f_2 ,\ \text{in}\  \mathbb{R}_+^2, \label{naa12}
\ee
\be
    \partial_n w_2-\nu w_2=0,\ \text{on}\    \Gamma. \label{naa13}
\ee
Below we demonstrate that both $w_1e^{-\nu x_2}$ and $w_2$ satisfy the estimate \eqref{Aoc}, but due to different reasons.

The  estimate for $w_2$ follows from the next lemma, which we prove under less restrictive conditions on $f$.
\begin{lemma}  Let $(x_2+1)f_2\in L_2(\mathbb{R}_+^2)$ and solution $w_2$ of the boundary value problem \eqref{naa12}, \eqref{naa13} satisfies \eqref{razl}. Then
\be  \|w_2\|^2_{\dot{H}^2(\mathbb{R}_+^2)}\leq c \frac{1+\nu^2}{\nu^2}\|(x_2+1)f_2\|^2_{ L_2(\mathbb{R}_+^2)},
\label{S1}\ee
where $c$ does not depend on $\ve$.
\end{lemma}
\begin{proof} Let us write down the ``energy" identity for the problem \eqref{naa12},  \eqref{naa13}:
\be     \int_{\mathbb{R}^2_+} |\nabla w_2|^2dx-\nu \int_{-\infty}^{+\infty}|w_2(x_1,0)|^2dx_1=-\mbox{Re} \int_{\mathbb{R}^2_+} f_2 \overline{w_2}dx.\label{energy}
\ee
Since $w_2$ satisfies \eqref{razl} we have:
\be 2\nu|w_2(x_1,0)|^2\leq \int^{+\infty}_0|\partial_{x_2} w_2(x_1,x_2)|^2dx_2. \label{vps}
\ee
Now we deduce from \eqref{vps} and \eqref{energy}
\be      \int_{\mathbb{R}^2_+} |\nabla w_2|^2dx\leq-2\mbox{Re} \int_{\mathbb{R}^2_+} f_2 \overline{w_2}dx.  \label{ocenergy}
\ee
In order to estimate  right hand side of \eqref{ocenergy}  we employ the Hardy type inequality, cf e.g. \cite{Hardy},
\be
\label{H1} \int_0^{+\infty} (x_2+1)^{-2} |v|^2 dx_2\leq c\left(\int_0^{+\infty}  |\partial_{x_2}v|^2dx_2  +|v(0)|^2 \right) . \ee
It follows from \eqref{ocenergy}  \eqref{H1} and \eqref{vps} that, for any $\delta>0$,
\be \int_{\mathbb{R}^2_+} |\nabla w_2|^2dx\leq\delta ^{-1} \int_{\mathbb{R}^2_+}(x_2+1)^{2}  |f|^2 dx+\delta \int_{\mathbb{R}^2_+} (x_2+1)^{-2} |w_2|^2dx\leq   \label{ocenergy2}
\ee
$$\delta ^{-1} \int_{\mathbb{R}^2_+}(x_2+1)^{2}  |f|^2 dx+\delta c\int_{\mathbb{R}^2_+}  |\partial_{x_2}w_2|^2dx_2  +\delta c \int_{-\infty}^{+\infty}|w_2(x_1,0)|^2 dx_1\leq
$$
$$\delta ^{-1} \int_{\mathbb{R}^2_+}(x_2+1)^{2}  |f|^2 dx+\delta c\int_{\mathbb{R}^2_+}  |\partial_{x_2}w_2|^2dx_2  +\frac{\delta c}{2\nu}\int_{\mathbb{R}^2_+}|\partial_{x_2}w_2|^2 dx.
$$
This implies,
\be \int_{\mathbb{R}^2_+} |\nabla w_2|^2dx\leq c\int_{\mathbb{R}^2_+}(x_2+1)^{2}  |f_2|^2 dx \label{ocenergy3},
\ee
\be \int_{-\infty}^{+\infty}|w_2(x_1,0)|^2dx_1\leq c \int_{\mathbb{R}^2_+}(x_2+1)^{2}  |f_2|^2 dx \label{ocenergy5},
\ee
\be \int_{\mathbb{R}^2_+} (x_2+1)^{-2} |w_2|^2dx\leq c\int_{\mathbb{R}^2_+}(x_2+1)^{2}  |f_2|^2 dx \label{ocenergy6}.
\ee
>From the boundary value problem \eqref{naa12}, \eqref{naa13} we further have, using standard elliptic estimates and suitable cut-off functions,
\be \int_{\mathbb{R}^2_+}  |\nabla^2 w_2|^2dx\leq c\int_{\mathbb{R}^2_+}(x_2+1)^{2}  |f_2|^2 dx. \label{ocenergy7}
\ee
The estimates \eqref{ocenergy3}, \eqref{ocenergy6} and \eqref{ocenergy7} imply \eqref{S1}.
\end{proof}
The estimate for $w_1e^{-\nu x_2}$ is ensured by the   following lemma, which we formulate in a self-contained form.
\begin{lemma}  Consider one-dimensional Schr\"{o}dinger equation, cf \eqref{hhh},
\be \partial_{t}^2 u(t) + (\nu^2-i\ve) u (t)= f(t), \ \ t\in \mathbb{R},\label{hhhh}
\ee
with  absorption $\ve\in\mathbb{R}$,  $|\ve|<1$, and rapidly decaying right hand side $f$.
More precisely we assume that  $e^{\delta\langle t \rangle}f \in L_2(\mathbb{R})$ for some $\delta >\frac{|\ve|}{2\nu}$, $\langle t\rangle=(t^2+1)^{1/2}$. If additionally
$e^{\delta\langle t \rangle}u\in H^2(\mathbb{R})$   then, with constant $c$ independent of $\nu$ and $\ve$,
 \be \frac{\nu}{1+\nu^2}\|u_{tt}\|_{L_2}+\|u_t\|_{L_2}+\nu \|u\|_{L_2}\leq c \left(\int_{-\infty}^{+\infty} |f|^2\langle t \rangle^2e^{\frac{|\ve|}{\nu}\langle t \rangle}dt\right)^{1/2}.\label{S2}
\ee
\end{lemma}
\begin{proof} Let us scalar multiply \eqref{hhhh} by $te^{\alpha| t |}u_t$, where $\alpha=\frac{|\ve|}{\nu}$, and integrate by parts. We have
$$ -2\mbox{Re}\int_{-\infty}^{+\infty} f(t)te^{\alpha| t |}\overline{u_t}dt=$$
$$-2\mbox{Re}\int_{-\infty}^{+\infty}\partial_{t}^2 ute^{\alpha|t|}\overline{u_t}dt - 2\mbox{Re}\left[(\nu^2-i\ve)\int_{-\infty}^{+\infty} u te^{\alpha|t|}\overline{u_t}dt\right]
$$
$$=\int_{-\infty}^{+\infty} \left(te^{\alpha|t|}\right)_t|u_t|^2dt + \nu^2\int_{-\infty}^{+\infty} \left(te^{\alpha|t|}\right)_t|u |^2tdt+2\mbox{Re}\left[i\ve\int_{-\infty}^{+\infty} u te^{\alpha|t|}\overline{u_t}dt\right]=
$$
$$\int_{-\infty}^{+\infty}e^{\alpha|t|}\left(1  +\alpha |t|\right)|u_t|^2dt + \nu^2\int_{-\infty}^{+\infty} e^{\alpha|t|}\left(1  +\alpha |t|\right)|u |^2dt+2\mbox{Re}\left[i\ve\int_{-\infty}^{+\infty}  te^{\alpha|t|}u\overline{u_t}dt\right]
$$
$$=\int_{-\infty}^{+\infty}e^{\alpha|t|} |u_t|^2dt + \int_{-\infty}^{+\infty}\frac{\alpha }{| t |}e^{\alpha|t|} \left|t u_t+i\frac{\ve}{\alpha}| t | u\right|^2dt
+ \nu^2\int_{-\infty}^{+\infty}e^{\alpha|t|}\left|u\right|^2dt.
$$
From the above we conclude
\be \int_{-\infty}^{+\infty}e^{\alpha|t|} |u_t|^2dt \leq 4 \int_{-\infty}^{+\infty} |f|^2t^2e^{\alpha|t|}dt,\label{ocyy}
\ee
and
\be \nu^2 \int_{-\infty}^{+\infty}e^{\alpha|t|} \left|u\right|^2dt\leq  \int_{-\infty}^{+\infty} |f|^2t^2e^{\alpha|t|}dt.\label{ocyy3}
\ee
Second derivatives can be estimated directly via
 equation \eqref{hhhh} and  we obtain
\be \|u_{tt}e^{\frac{\alpha}{2}|t|}\|_{L_2(\mathbb{R})}\leq \frac{\nu^2+|\ve|}{\nu}\|tfe^{\frac{\alpha}{2}|t|}\|_{L_2(\mathbb{R})}+\|fe^{\frac{\alpha}{2}|t|}\|_{L_2(\mathbb{R})}.
\label{ocyyy2}\ee

Combining \eqref{ocyy}, \eqref{ocyy3}  and \eqref{ocyyy2} we finally obtain  \eqref{S2}.
\end{proof}
Finally we obtain the estimate for $\tilde{u}_\ve$. From \eqref{razl0},\eqref{S2},\eqref{S1} and \eqref{ttttt}  we have
\be  \|\tilde{u}_\ve\|_{\dot{H}^2}\leq  \|w_1e^{-\nu x_2}\|_{\dot{H}^2}+\|w_2\|_{\dot{H}^2}\leq c\|f\langle x_1 \rangle e^{\frac{\ve}{2\nu}\langle x_1 \rangle}\|_{L_2}+c\|\langle x_2 \rangle f_2\|_{ L_2(\mathbb{R}_+^2)}\lesssim \|F\|_{W^0_{\beta,1}}.
\ee
Now estimate \eqref{tttt} delivers the desired result \eqref{ocep3}.


\end{document}